\documentclass[10pt,oneside]{amsart} %added oneside so comments stay on the right, can remove when finished

\usepackage{graphicx} % Required for inserting images

\usepackage{amsmath,amstext,amscd, amssymb,txfonts, epsfig, color, epstopdf}
\usepackage{pinlabel}
\usepackage[normalem]{ulem}
\usepackage[colorlinks=true, pdfstartview=FitV, linkcolor=blue, citecolor=blue, urlcolor=blue]{hyperref}
 \usepackage[abs]{overpic}
\usepackage{xcolor,varwidth}
\usepackage{enumerate}
\usepackage{tabularx}
\usepackage{tikz}
 \usepackage{enumitem}
 \usepackage{csquotes}
\makeatletter
\def\@tocline#1#2#3#4#5#6#7{\relax
\ifnum #1>\c@tocdepth % then omitf
  \else 
    \par \addpenalty\@secpenalty\addvspace{#2}% 
\begingroup \hyphenpenalty\@M
    \@ifempty{#4}{%
      \@tempdima\csname r@tocindent\number#1\endcsname\relax
 }{%
   \@tempdima#4\relax
 }%
 \parindent\z@ \leftskip#3\relax \advance\leftskip\@tempdima\relax
 \rightskip\@pnumwidth plus4em \parfillskip-\@pnumwidth
 #5\leavevmode\hskip-\@tempdima #6\nobreak\relax
 \ifnum#1<0\hfill\else\dotfill\fi\hbox to\@pnumwidth{\@tocpagenum{#7}}\par
 \nobreak
 \endgroup
  \fi}
\makeatother
\let\oldtocsection=\tocsection
\let\oldtocsubsection=\tocsubsection
\let\oldtocsubsubsection=\tocsubsubsection
\renewcommand{\tocsection}[2]{\hspace{0em}\oldtocsection{#1}{#2}}
\renewcommand{\tocsubsection}[2]{\hspace{1em}\oldtocsubsection{#1}{#2}}
\renewcommand{\tocsubsubsection}[2]{\hspace{2em}\oldtocsubsubsection{#1}{#2}}
\definecolor{cerulean}{rgb}{0,.48,.65} 
\definecolor{magenta}{rgb}{.5,0,.5} 
\definecolor{dred}{rgb}{.5,0,0} 
\definecolor{green}{rgb}{0,.5,0} 
\definecolor{blue}{rgb}{0,0,1} 
\definecolor{black}{rgb}{0,0,0} 
\definecolor{dgreen}{rgb}{0,.3,0} 
\definecolor{vdred}{rgb}{.3,0,0} 
\definecolor{red}{rgb}{1,0,0} 
\definecolor{salmon}{rgb}{0.98,0.50,0.45} 
\definecolor{gray}{rgb}{.5,.5,.5} 
\definecolor{seagreen}{rgb}{0.13,0.70,0.67} 
\definecolor{chartreuse}{rgb}{0.40,0.80,0.00}
\definecolor{cornflower}{rgb}{0.39,0.58,0.93} 
\definecolor{gold}{rgb}{0.80,0.68,0.00}
 
\theoremstyle{plain}
 \newtheorem{MainThm}{Theorem}

\newtheorem{thm}{Theorem}[section]
\newtheorem{lemma}[thm]{Lemma}
\newtheorem{claim}[thm]{Claim}

\newtheorem{cor}[thm]{Corollary}
\newtheorem{prop}[thm]{Proposition}

\newtheorem{Conjecture}[thm]{Conjecture}

\theoremstyle{definition}
\newtheorem{rem}[thm]{Remark}
\newtheorem{defn}[thm]{Definition}

\newtheorem{fact}[thm]{Fact}

\newtheorem{Open questions}[thm]{Open questions}
\newtheorem{Open question}[thm]{Open question}
\newtheorem{Open problems}[thm]{Open problems}
\newtheorem{Open problem}[thm]{Open problem}

\def\ni{\noindent}

\def\Dist{\hbox{\rm Dist}}

\def\CL{\hbox{\rm CL}}

\def\F+L{\hbox{$\textup{F}\!_+\textup{L}$}}

\newcommand{\BS}{\mathrm{BS}}

\def\onto{{\kern3pt\to\kern-8pt\to\kern3pt}}

\def\<{\langle}
\def\>{\rangle}
\def\|{{\ |\ }}

\def\a{\alpha}
\def\b{\beta}

\newcommand{\abs}[1]{\left|#1\right|}
\renewcommand{\ni}{\noindent}

\def\*{^{\star}}

\newcommand{\NN}{\mathbb{N}}

\newcommand{\ZZ}{\mathbb{Z}}
\newcommand{\lcm}{\text{lcm}}
\newcommand{\del}{\partial}

\title{Conjugator Length in  Iterated Cyclic Amalgams of Free Groups}
\author{Conan Gillis}
\date{September 2026}

\begin{document}

\maketitle
\begin{abstract}
Let $\mathcal{A}_0$ be the set of finitely generated free groups on some countable alphabet and and define $\mathcal{A}_r$ inductively as the union of $\mathcal{A}_{r-1}$ with the set of amalgamated free products $G_1\ast_{\ZZ}G_2$, where $G_1,G_2\in \mathcal{A}_{r-1}$. We show, for all $r$, that every element of $\mathcal{A}_r$ has linear conjugator length, quantifying a result of Larsen and extending a result of Kharlampovich and Myasnikov. This is achieved by a linear conjugator length bound on a family of iterated HNN extensions of free groups. Also, in an appendix, we adapt a proof of A. Wei$\ss$ to show that all Generalized Baumslag-Solitar Groups have linear conjugator length. As a consequence, we obtain that one-relator groups with non-trivial center have linear conjugator length, giving evidence for a conjecture of Bridson, Riley, and Sale.
\end{abstract}

\section{Introduction} Much recent work has focused on quantifying the difficulty of Max Dehn's three famous problems for specific finitely presented groups -- the word and conjugacy problems have received particular attention due (in part) to their potential applications in public-key cryptography. These problems can be studied in terms of their ``external" difficulty, e.g. algorithmic complexity, as well as the ``internal" difficulty. The later approach uses a group's intrinsic geometry to quantify the complexity of solutions to the search-variants of Dehn's questions; the word problem for a group $G$, for example, is quantified by its smallest isoperimetric function, otherwise known as its Dehn function (see \cite{Bridson6}). The present paper considers the conjugator length function, which quantifies the conjugacy problem and to which we presently turn.

\begin{defn}
    Let $G=\langle X\mid R\rangle$ be a finitely presented group, and $|g|=\min\{n: g=g_1g_2\cdots g_n\text{ for }g_i\in X^{\pm1}\}$. For each pair of conjugate elements $u,v\in G$, we define $$c(u,v)=\min\{|\gamma|:\gamma u\gamma^{-1}=v\}$$ and the \textit{conjugator length function} $\CL_G$ by $$\CL_G(n)=\max c(u,v),$$ where the maximum is taken over all conjugate pairs $u,v$ such that $|u|+|v|\leq n$. 
\end{defn}

Of course, all three of $|\cdot|,c(\cdot,\cdot)$, and $\CL_G(\cdot)$ depend on the presentation chosen, however it is easily shown that the conjugator length function for $G$ is well defined up to a standard equivalence relation $\simeq$ (see Definition \ref{equivRel}).

Our first result concerns certain iterated HNN extensions of free groups. We adopt, once and for all, the convention that all cyclic groups are infinite, and thus isomorphic to $\ZZ$.
\begin{MainThm}\label{SecondMainTheorem}
    Suppose $G$ has a presentation of the form $\langle \Lambda_0, t_1,\ldots, t_M\mid t_ix_it_i^{-1}=y_i\rangle$ where \begin{enumerate} \item\label{Introxiyirelators}$\Lambda_0$ is a finite set and $x_i,y_i$ are words on $\Lambda_0\cup\{t_1,\ldots, t_{i-1}\}$ not representing the identity,  \item\label{Introlambdafree}  $\langle\Lambda_0\rangle_{G}$ is free with free basis $\Lambda_0$, and \item\label{IntronoDistortion} $\langle x_i\rangle_{G},\langle y_i\rangle_{G}$ are both undistorted in $G$.\end{enumerate} Then $\CL_G(n)\simeq n$.
\end{MainThm}
Item \eqref{IntronoDistortion} is essential here: the Baumslag-Gersten group $BG=\langle a,s,t\mid sas^{-1}=a^2, tat^{-1}=s\rangle$ has a massively distorted subgroup, $\langle a\rangle_{BG}$ \cite{Gersten1992,platonov2004isoparametric}, and has conjugator length function  equivalent to a power tower of height $\lfloor \log_2 n\rfloor$ \cite{gillis2025conjugator}.

Theorem \ref{SecondMainTheorem} gives quantification of \cite[Theorem 2]{Larsen_1977} and of the main result of \cite{bezverkhnii2016conjugacy}, which is the main motivation of this paper. To set notation, following \cite{Larsen_1977}, let $\mathcal{A}_0$ be the set of finitely generated free groups with free basis contained in the set of letters $\Lambda=\{\lambda_1,\lambda_2,\ldots\}$, and define $\mathcal{A}_r$ inductively as the union of $\mathcal{A}_{r-1}$ with the set of amalgamated free products $G_1\ast_{C}G_2$, where $G_1,G_2\in \mathcal{A}_{r-1}$ and $C$ is a cyclic group embedding into $G_1$ and $G_2$. We will show (Proposition \ref{iterFundConjPowers}) that every element of $\mathcal{A}_r$ is a retract of a group $\mathcal{F}(G)$ satisfying Theorem \ref{SecondMainTheorem}, whence we obtain:

\begin{cor}\label{AkCor}  For all $r$ and all $G\in \mathcal{A}_r$, $\CL_G(n)\simeq n$.\end{cor}  This Corollary presents an interesting contrast with the computations in  \cite{gillis2025conjugator}, which indicate that iterated cyclic HNN-extensions of free groups need not have linear conjugator length functions. A final note on this result: Kharlampovich and Myasnikov \cite{kharlampovich1998hyperbolic} showed that iterated cyclic amalgams of hyperbolic groups are hyperbolic if only if (among other conditions) all edge-groups are maximal in one of their corresponding vertex-groups -- our result extends this in the case where the base groups are free by showing that maximality is not necessary for the amalgams to have linear conjugator length function, even if hyperbolicity would be lost. Building on this, we make the following conjecture:

\begin{Conjecture}
All iterated cyclic amalgams of hyperbolic groups have linear conjugator length functions.\end{Conjecture}
The techniques of this paper could likely be generalized to answer this conjecture; indeed, replacing freeness with torsion-free hyperbolicity in Condition \ref{Introlambdafree} of Theorem \ref{SecondMainTheorem} and in Subsections \ref{IFGdef} and \ref{IFGRedux}, then showing Theorem \ref{SecondMainTheorem} still holds, would suffice for all torsion-free hyperbolic groups (torsion would add additional considerations). This would require additional tools from hyperbolic group theory, however, as we use freeness in a crucial way here.

In an appendix to this paper we also prove:

\begin{MainThm}\label{GBSMain}
    Let $\Gamma$ be a graph of cyclic groups, and let $G$ be the Generalized Baumslag-Solitar group defined by $\Gamma$. We have $\CL_G(n)\simeq n$.
\end{MainThm}

During the preparation of this paper, we initially showed Theorem \ref{GBSMain} for the case where $\Gamma$ has Betti-number 1 using the techniques developed below, however the author's work with a Large Language Model (Claude Fable) gave a much shorter and more general proof. Combined with Corollary \ref{AkCor} and \cite[Theorem 1]{Pietrowski1974}, Theorem \ref{GBSMain} (even restricted to the Betti-number-equals-one case) implies one of the main motivations of this paper:
\begin{cor}\label{OneRelCor}
    For all one-relator groups $G$ whose center is non-trivial, $\CL_G(n)\simeq n$. 
\end{cor}
Baumslag and Taylor \cite{BAUMSLAGTAYLOR} showed that these groups have an additional description as a semidirect product $F\rtimes_\varphi\ZZ$, where $F$ is some finitely-generated free group and $\varphi$ has finite Out-order. Our result thus gives partial evidence for a conjecture of Bridson, Riley, and Sale \cite{BridsonRileySaleMMJ} that all free-by-cyclic groups have linear conjugator length function.

\subsection{Acknowledgments} The author thanks Tim Riley for his very helpful advice and comments. This work was generously supported by NSF Grant DGE–2139899.

\subsection{AI Usage} The Claude LLMs Opus and Fable were used to proofread for typographical and mathematical errors and to prove Theorem \ref{GBSMain}. Due to the difference in methods, this proof is reserved to an appendix. The author wishes it to be known that, to the best of his knowledge, the key ideas in the elliptic case of Theorem \ref{GBSMain}'s proof are due to A. Wei$\ss$ and the works cited in \cite{Weiss}. Despite this, he still thinks this result should be recorded in the literature, and is of direct relevance to the body the present paper. Of course, the author has written all of the text by hand and accepts full responsibility for the paper's contents.
\section{Preliminaries}
\subsection{Notation and Conventions}
Our main equivalence relation is:
\begin{defn}\label{equivRel} For non-decreasing functions $f,g:\NN\to \NN$, we write $f(n)\preceq g(n)$ if there exists $C>0$ such that $f(n)\leq Cg(Cn+C)+Cn+C$ for all $n\in \NN$, and $f(n)\simeq g(n)$ if $f(n)\preceq g(n)$ and $g(n)\preceq f(n)$.
    
\end{defn}

Also, since each element of $\mathcal{A}_r$ is countable for all $r$, and $\mathcal{A}_0$ is countable, we see by induction that $\mathcal{A}_r$ is countable for all $r$. For the rest of this paper, and for all $r>0$ and $G\in \mathcal{A}_r\smallsetminus\mathcal{A}_{r-1}$, we may therefore fix a splitting $G=A\ast_CB$ where $C$ is cyclic and $A,B\in \mathcal{A}_{r-1}$, without appeal to set-theoretic techniques. We further assume without loss of generality that $A$ and $B$ are always disjoint before amalgamating, and that all elements of $\mathcal{A}_0$ are generated by a subset of the same alphabet $\Lambda$.

Finally, we recall that in a finitely presented group $G=\langle X\mid R\rangle$ every element $g\in G$ is equal to a (not necessarily unique) sequence $e_1e_2\ldots e_m$ of elements $e_i$ of $X\cup X^{-1}$; we call a sequence of this sort a word on $X$ which represents $g$, and call each term in the sequence a letter. We will frequently elide the difference between words and the elements they represent, saying e.g. ``the order of the word $u$ on $X$ is 5" in place of ``the word $u$ on $X$ represents an element of $G$ which has order 5." The word-length $|g|$ of an element $g$ with respect to the presentation $\langle X\mid R\rangle$ is the length of its shortest representative. Also, sometimes $X$ will contain a finite set of distinguished elements, always denoted $t_1,\ldots, t_M$. If $e_i=t_j$ in some word $e_1e_2\ldots e_m$ on $X$, we will call that letter either a $t$-letter or (for specificity) a $t_j$-letter. We will denote the number of instances of $t_i$ and $t_i^{-1}$ in a word $w$ by $\#_iw$.
\section{Graphs of Groups} \subsection{HNN-Extensions} Introduced by Graham Higman and Hanna and Bernhard Neumann in \cite{HNNOriginal}, HNN extensions are a prototypical special case of fundamental groups of graphs of groups, which we treat in the next subsection.

Given a finitely presented group $G=\langle X\mid R\rangle$ and two subgroups $A,B\leq G$ with an isomorphism $\varphi:A\to B$, the HNN extension $G'$ along $\varphi$ is the group $\langle X,t\mid R, tat^{-1}=\varphi(a)\ \forall a\in A\rangle $; the generator $t$ is called the stable letter. We record here three essential facts that will be of immediate use; for proofs and a more general treatment, we refer the reader to \cite{rotman1999introduction}.

\begin{lemma}[HNN Embedding Lemma]
   Let $G,G'$ be as above. Then $\langle X\rangle_{G'}$ is isomorphic to $G$ by the map induced by the identity on $G$. 
\end{lemma}
In the sequel, we will simply denote $\langle X\rangle_{G'}$ by $G$, writing $G\leq G'$ to when ambiguity would occur.
\begin{lemma}[Britton's Lemma] Let $G=\langle X\mid R\rangle$ and subgroups $A,B\leq G$ with an isomorphism $\varphi:A\to B$. Let $G'$ be an HNN-extension of $G$ along $\varphi$ with stable letter $t$, and let $w$ be a word on $X\cup\{t\}$. If $w$ represents an element of $G\leq G'$ and contains at least one $t$-letter, then $w$ contains a subword of one of the following forms:
\begin{itemize}
    \item $txt^{-1}$, where $x$ is equal (in $G$) to $a\in A$, or
    \item $t^{-1}yt$, where $y$ is equal (in $G$) to $b\in B$.
\end{itemize}
    
\end{lemma}
We call a word of these two forms a $t$-pinch.
\begin{cor}\label{BrittonCor}
    Let $G,G'$ be as above, and $w$ a word on $X\cup\{t\}$. If no conjugate of $w$ is an element of $G\leq G'$, then $w$ is conjugate to a cyclically freely reduced word $w'$ on $X\cup\{t\}$ such that $w'$ contains at least one $t$-letter and no cyclic permutation of $w'$ contains any $t$-pinches. 
\end{cor}
\begin{proof}
    If some conjugate of $w$ has no $t$-pinches in any cyclic conjugate and no $t$-letters, $w$ is conjugate to an element of $G$. Supposing this is not the case, if some cyclic conjugate $w'$ of $w$ has both $t$-letters and $t$-pinches, then replacing the pinch $txt^{-1}$ with $\varphi(x)$ (where $x=a\in A$) or the pinch $t^{-1}yt$ with $\varphi^{-1}(y)$ (where $y= b\in B$), as the case may be, strictly reduces the number of $t$-letters in \textit{any} cyclic conjugate of $w'$ while leaving at least one $t$-letter. Repeating this process, cyclically conjugating as needed, gives $w'$.
\end{proof}

\subsection{Fundamental Groups of Graphs of Groups}\label{FundGrps} This discussion is taken from the classic reference \cite{trees}. While we have explicated certain details that are crucial to our purpose, we claim no originality in this subsection. A graph $\Gamma$ consists of two disjoint sets $V(\Gamma),E(\Gamma)$, along with a map $E(\Gamma)\to V(\Gamma)^2$ denoted by $e\mapsto (\alpha(e),\omega(e))$ and an involution $E(\Gamma)\to E(\Gamma)$ denoted $e\mapsto \overline{e}$ such that $\alpha(\overline{e})=\omega(e)$ and $\omega(\overline{e})=\alpha(e)$. A reduced path in a graph is a sequence of edges $e_1,e_2,\ldots, e_m$ with $\omega(e_i)=\alpha(e_{i+1})$ and $e_i\neq \overline{e_{i+1}}$, and a circuit is a reduced path with $\alpha(e_1)=\omega(e_m)$; a tree is a graph with no circuits. 

Let $\Gamma$ be a finite graph. We associate to each vertex $v\in V(\Gamma)$ a group $G_v$ and to each edge $e\in E(\Gamma)$ with $u=\alpha(e),v=\omega(e)$ a group $H_e$ and two injections $\sigma_e:H_e\hookrightarrow G_u,\tau_e:H_e\hookrightarrow G_v $ such that $\sigma_{\overline{e}}=\tau_e$ and $\tau_{\overline{e}}=\sigma_e$. We encapsulate all of this data in the tuple $$\mathcal{G}=\left ((G_{v_1},\ldots, G_{v_k}),(H_{e_1},\ldots, H_{e_\ell}) ,(\sigma_{e_{1}},\ldots, \sigma_{e_\ell}),(\tau_{e_{1}},\ldots, \tau_{e_\ell})\right ),$$ and call the pair $(\Gamma, \mathcal{G})$ a \textit{graph of groups}. For our purposes, $H_e$ will always be cyclic and we will generally denote the images of the generator of $H_e$ under$\sigma_e$ and $\tau_e$ by $a_e$ and $b_e$ respectively, so $a_{\overline{e}}=b_e$. Also, if $\Gamma'$ is a subgraph of $\Gamma$, then the \textit{restriction} $\mathcal{G}'$ of $\mathcal{G}$ to $\Gamma'$ is the result of removing $G_v,H_e,\sigma_e,\tau_e$ from their respective components of $\mathcal{G}$, for all $v\in V(\Gamma)\smallsetminus V(\Gamma'),e\in E(\Gamma)\smallsetminus E(\Gamma')$.

We abbreviate the set $\{t_e:e\in E(\Gamma)\}$ as $\{t_e\}$, the set $\left\{t_{\overline{e}}t^{-1}_e:e\in E(\Gamma)\right\}$ as $\left\{t_{\overline{e}}t^{-1}_e\right\}$, and the set $\left\{t_ea_et_e^{-1}b_e^{-1}:E(\Gamma)\right\}$ as $\left\{t_ea_et_e^{-1}b_e^{-1}\right\}.$ If each vertex group has the presentation $G_{v}=\langle X_v\mid R_v\rangle$, then the \textit{fundamental groupoid} $\mathcal{F}(\Gamma, \mathcal{G})$ has the presentation $$\mathcal{F}(\Gamma, \mathcal{G})=\left\langle \bigsqcup_{v\in V(\Gamma)} X_v\cup\{t_e\}\ \middle\vert\  \bigsqcup_{v\in V(\Gamma)} R_v\cup\left\{t_{\overline{e}}t_{e}^{-1}\right\}\cup\left\{t_ea_et_e^{-1}b_e^{-1}\right\}  \right\rangle.$$ ($\mathcal{F}(\Gamma, \mathcal{G})$ is, to be clear, a group, however this nomenclature reflects its topological inspiration and is given in \cite{trees}.) For a spanning tree $T$ of $\Gamma$, the \textit{fundamental group with respect to $T$}, denoted $\pi_1(\Gamma,\mathcal{G},T)$, is the quotient of $\mathcal{F}(\Gamma,\mathcal{G})$ by the normal subgroup generated by $\{t_e:e\in E(T)\}$. 

\begin{rem}
    In the event $\Gamma$ is a single edge $e$ between $u$ and $v$, the fundamental group will be the amalgamated free product $G_u\ast_{H_e}G_v$ -- the presentation above will be our default presentation for any such product.
\end{rem}

\begin{rem}\label{fundGroupoidInclusion}Up to isomorphism, the fundamental group does not depend on $T$ \cite[Proposition 5.1.20]{trees}; indeed, if $\rho:\mathcal{F}(\Gamma,\mathcal{G})\twoheadrightarrow  \pi_1(\Gamma,\mathcal{G},T)$ is the above quotient, there exists an injection $i: \pi_1(\Gamma,\mathcal{G},T)\hookrightarrow  \mathcal{F}(\Gamma,\mathcal{G})$ such that $\rho\circ i$ is the identity and, for some other spanning tree $T'$ and corresponding quotient $\rho':\mathcal{F}(\Gamma,\mathcal{G})\twoheadrightarrow  \pi_1(\Gamma,\mathcal{G},T')$, the map $\rho'\circ i$ is an isomorphism. The proof of this same Proposition indicates $i(X_v)\subseteq \mathcal{F}(\Gamma,\mathcal{G})$ is a conjugate of $X_v\subseteq \mathcal{F}(\Gamma,\mathcal{G})$ by a sequence of $t$-letters.\end{rem}

\begin{defn}\label{treeprod}
    
If $\Gamma$ is a tree, $T$ is unique and $E(\Gamma)=E(T)$. We call the fundamental group of $(\Gamma,\mathcal{G})$ the \textit{tree product} of the vertex groups, denote it by $\pi_1(\Gamma,\mathcal{G})$, and (abbreviating $\{a_eb_e^{-1}:e\in E(\Gamma)\}$ by $\{a_eb_e^{-1}\}$) note that it has the presentation $$\pi_1(\Gamma, \mathcal{G})=\left\langle \bigsqcup_{v\in V(\Gamma)} X_v\ \middle\vert\  \bigsqcup_{v\in V(\Gamma)} R_v\cup\{a_eb_e^{-1}\}  \right\rangle.$$ In particular, if $\Gamma$ consists of a single edge $H_e$ and two endpoints $u\neq v$, then $\pi_1(\Gamma, \mathcal{G})$ is the amalgamated free product $G_u\ast_{H_e}G_v$ of $G_u,G_v$ along $H_e$ (if $u=v$ instead, then $\pi_1(\Gamma, \mathcal{G})$ is an HNN extension of $G_u=G_v$ associating $a_e$ to $b_e$). Finally, for any vertex group $G_v$, the obvious map $G_v\to \langle X_v\rangle_{\pi_1(\Gamma,\mathcal{G})}$ is injective, and is called the canonical inclusion.
\end{defn}

We conclude this section with a universal property of amalgamated free products.

\begin{lemma}[\cite{trees}]\label{univAmalgams}
    Let $G=A\ast_C B$, with $x$ and $y$ the generators of $C$ in $A$ and $B$ respectively. Given maps $f_1:A\to H,f_2:B\to H$ such that $f_1(x)=f_2(y)$, there exists a unique map $f:G\to H$ with $f|_A=f_1,f|_B=f_2$.
\end{lemma} 

\begin{cor}\label{rightInvUniv}
    Let $G=A\ast_C B$, with $x$ and $y$ the generators of $C$ in $A$ and $B$ respectively. Given maps $f_1:A\to A',f_2:B\to B'$, suppose there exist $g_1:A'\to A,g_2:B'\to B$ such that $g_i\circ f_i$ is the identity and let $H=A'\ast_{C'}B'$ where the generators of $C'$ in $A'$ and $B'$ are $f_1(x)$ and $f_2(y)$ respectively.  Then there exist maps $f:G\to H$ and $g:H\to G$ such that $g\circ f$ is the identity.
\end{cor}

\begin{proof}
    The existence of some $f:G\to H$ is given by Lemma \ref{univAmalgams}, since $f_1(x)=f_2(y)$ in $H$. For $g:H\to G$, let $i_1:A\hookrightarrow G,i_2:B\hookrightarrow G$ be the canonical inclusions. In $G$, we have $i_1(g_1(f_1(x)))=x=y=i_2(g_2(f_2(y)))$, hence there exists a map $g:H\to G$ with $g|_{A'}=i_1\circ g_1$ and $g|_{B'}=i_2\circ g_2$. Note that $g\circ f|_A=g\circ f_1$ fixes every element of $A$, and $g\circ f|_B=g\circ f_2$ fixes every element of $B$, hence $g\circ f$ is the identity on $G$ as desired.
\end{proof}
Note in this construction that $f|_A$ equals $f_1$ composed with the inclusion $A'\hookrightarrow H$ and $f|_B$ equals $f_2$ composed with the inclusion $B'\hookrightarrow H$.

\subsection{Iterative Fundamental Groupoids}\label{IFGdef}
In this subsection we construct the iterative fundamental groupoid of $G\in \mathcal{A}_r$, which is a group $\mathcal{F}(G)$ and an injection $f_G:G\hookrightarrow \mathcal{F}(G)$ such that there exists a left inverse $f'_G:\mathcal{F}(G)\twoheadrightarrow G$ to $f_G$.
\begin{defn}\label{defIterGroupoid} We construct $\mathcal{F}(G),f_G$ inductively, based on the splittings of each element of $\mathcal{A}_r$ that we fixed above: for $G\in \mathcal{A}_0$, let $\mathcal{F}(G)=G$ and $f_G=id_G$. For $G=A\ast_CB$, where the generators of $C$ in $A,B\in \mathcal{A}_{r-1}$ are $x$ and $y$ respectively, let $G'=\mathcal{F}(A)\ast_{C'}\mathcal{F}(B)$, where the generators of $C'$ in $A$ and $B$ are $f_A(x)$ and $f_B(y)$; Corollary \ref{rightInvUniv}, when applied to $f_A$ and $f_B$, gives maps $f_0:G\to G'$, $g_0:G'\to G$. Next, let $\mathcal{F}(G)$ be the HNN extension of the free product $\mathcal{F}(A)\ast\mathcal{F}(B)$ with stable letter $t$ conjugating $f_A(x)$ to $f_B(y)$, and let $f_1:G'\hookrightarrow \mathcal{F}(G)$ be the homomorphism defined by $f_1(a)=tat^{-1}$ for $a\in \mathcal{F}(A)$, $f_1(b)=b$ for $b\in \mathcal{F}(B)$ (this exists by Lemma \ref{univAmalgams}). Letting $g_1:\mathcal{F}(G)\to G'$ be the quotient map killing $t$, we can let $f_G=f_1\circ f_0$ and $f_G'=g_0\circ g_1$, completing our definition.\end{defn}

\begin{rem}\label{CLRetract}
    Note that $G$ is a retract of $\mathcal{F}(G)$, whence $\CL_{G}(n)\preceq \CL_{\mathcal{F}(G)}(n)$. 
\end{rem}

\begin{lemma}\label{iterFundGrpPresentation}
    For all $r\geq 0$ and $G\in \mathcal{A}_r$, $\mathcal{F}(G)$ has a presentation of the form $\langle \Lambda_0, t_1,\ldots, t_M\mid t_ix_it_i^{-1}=y_i\rangle$ where $\Lambda_0$ is a finite subset of $\Lambda$ and $x_i,y_i$ are words on $\Lambda_0\cup\{t_1,\ldots, t_{i-1}\}$ not representing the identity. Moreover, $\langle\Lambda_0\rangle_{\mathcal{F}(G)}$ is free with free basis $\Lambda_0$, and each $y_i$ is conjugate to some element of $im(f_G)$.
\end{lemma}
\begin{proof}
    We proceed by induction on $r$. The base case is trivial (take $M=0$). If $G=A\ast_CB\in \mathcal{A}_1$, where $A=\langle \Lambda_1\rangle$ and $B=\langle\Lambda_2\rangle$ are free, then $\mathcal{F}(G)$ is the HNN extension of the free group $\langle \Lambda_1\rangle\ast \langle \Lambda_2\rangle$ associating the generator $x=x_1$ of $C$ in $A$ to the generator $y=y_1$ of $C$ in $B$; since $y_1\in im(f_G)$ by construction, the base case is done. 

    For the inductive step, let $G=A\ast_CB\in \mathcal{A}_r$ and suppose $\mathcal{F}(A)=\langle \Lambda_1, t_1,\ldots, t_M\mid t_ix_it_i^{-1}=y_i\ (i\leq M)\rangle$, $\mathcal{F}(B)=\langle \Lambda_2, t_{M+1},\ldots, t_{M'}\mid t_jx_jt_j^{-1}=y_j\ (M+1\leq j\leq M')\rangle$ satisfy the Lemma. Letting $x,y$ be the generators of $C$ in $A$ and $B$ as usual, we see $\mathcal{F}(G)$ is an HNN extension of $\mathcal{F}(A)\ast \mathcal{F}(B)$ with stable letter $t_{M'+1}$ associating $x_{M'+1}=f_A(x)$ to $y_{M'+1}=f_B(y)$. We thus have $$\mathcal{F}(G)=\langle \Lambda_1,\Lambda_2,t_1,\ldots, t_{M'}\mid t_ix_it_i^{-1}=y_i\ (i\leq M'), t_{M'+1}f_A(x)t_{M'+1}^{-1}=f_B(y)\rangle.$$ Setting $\Lambda_0=\Lambda_1\cup\Lambda_2$ establishes the desired presentation. As for $\langle \Lambda_0\rangle_{\mathcal{F}(G)}$ being free, note that $\langle \Lambda_1\rangle_{\mathcal{F}(A)}$ and $\langle \Lambda_2\rangle_{\mathcal{F}(B)}$ are free on the given generating sets, so $\Lambda_0$ is a free basis for a subgroup of $\mathcal{F}(A)\ast \mathcal{F}(B)\leq \mathcal{F}(G)$.

    Finally, to show that $y_i$ is conjugate into $im(f_G)$ for all $i\leq M'+1$, we have three cases: $i\leq M, M<i\leq M'$, and $i=M'+1$. If $i\leq M$, then $y_i$ is conjugate into $im(f_A)$ in $ \mathcal{F}(A)$; this means it is conjugate to an element of $im(f_0)$ in $G'$, since $f_0|_A=f_A$ by definition. Applying $f_1$, we have that $f_1(y_i)=ty_it^{-1}$ is conjugate to an element of $f_1(im(f_0))=im(f_G)$ in $\mathcal{F}(G)$, and the claim is immediate. The other two cases follow by similar and shorter arguments.
\end{proof}
\begin{cor}\label{iterBrittonForm}
    The subgroup $\langle \Lambda_0,t_1,\ldots ,t_i\rangle_{\mathcal{F}(G)}$ is an HNN extension of $\langle \Lambda_0,t_1,\ldots ,t_{i-1}\rangle_{\mathcal{F}(G)}$ with stable letter $t_i$.
\end{cor}
\begin{proof}
    This follows by induction on $M-i$, the presentation given in Lemma \ref{iterFundGrpPresentation}, the definition of an HNN extension, and the fact that an HNN extension of a torsion-free group is torsion free.
\end{proof}
 \begin{defn} We say a $t_i$-pinch is a word of the form $t_iwt_i^{-1}$, where $w$ represents $x_i^m$ for some $m\in\ZZ$, or of the form $t_i^{-1}wt_i$ where $w$ represents $y_i^m$ for some $m\in \ZZ$. A word $w$ on $\{\Lambda_0,t_1,\ldots, t_i\}$ is cyclically Britton-reduced with respect to $t_i$ if it is cyclically freely reduced and no cyclic permutation has any $t_j$-pinches for $j\leq i$. \end{defn}
 
 \begin{lemma}\label{cannonicalConjugates}
 For all $i\in \{1,\ldots, \ell\}$, if $x_i$ and $y_i$ are not conjugate to an element of $\langle \Lambda_0\rangle_{\mathcal{F}(G)}$, then there exist $j_i, k_i\in \{1,\ldots, \ell\}$ and some reduced words $x_i',y_i'$ on $\{\Lambda_0,t_1,\ldots, t_{j_i}\}$ containing a $t_{j_i}$- and $t_{k_i}$-letter respectively, which are conjugate respectively to $x_i,y_i$ and cyclically Britton-reduced with respect to $t_{j_i}$ or $t_{k_i}$.\end{lemma}\begin{proof} All save the final claim hold by Corollaries \ref{BrittonCor} and \ref{iterBrittonForm}. For the final claim, we need to remove $t_k$-pinches for $k\leq j_i,k_i$; this is done for $k=j_i$ also by Corollaries \ref{BrittonCor} and \ref{iterBrittonForm}. Downward induction on $k$ gives the $k<j_i, k_i$ case, using the fact that replacing $t_kx_k^mt_k^{-1}$ or $t_k^{-1}y_k^mt_k$ with $y_k^{m}$ or $x_k^m$ (resp.) decreases the number of $t$-letters.  
\end{proof}
% \subsection{One-Relator Groups} It was shown by Pietrowski in \cite{Pietrowski1974} that all one-relator groups with non-trivial center whose abelianization is not $\ZZ^2$ have a presentation of the form $$\langle a_1,\ldots, a_m\mid a_1^{p_1}=a_2^{q_1},\ldots, a_{m-1}^{p_{m-1}}=a_m^{q_{m-1}}\rangle,$$ and that one-relator groups with non-trivial center and abelianization isomorphic to $\ZZ^2$ have a presentation of the form $$\langle a,a_1,\ldots, a_m\mid a_1^{p_1}=a_2^{q_1},\ldots, a_{m-1}^{p_{m-1}}=a_{m}^{q_{m-1}}, aa_m^{p_m}a^{-1}=a_1^{q_m}\rangle.$$ Note that the first case is a GBS defined on a path and the second is a GBS defined on a cycle; both of these graphs have Betti-number at most 1.
\subsection{Subgroup Distortion}
We conclude this section with the definition of subgroup distortion.
\begin{defn}
    Let $H\leq G$ be two finitely presented groups; the distortion $\Dist_H^G(n)$ of $H$ in $G$ is the largest word-length in $H$ of any element $h\in H$ with $|h|\leq n$ (regardless of the chosen generating sets for $H$ and $G$, this is well defined up to $\simeq$). If $\Dist_H^G(n)\simeq n$, we say $H$ is \textit{undistorted in }$G$.
\end{defn}
We will mainly use the following three facts about subgroup distortion, recorded here for reference:
\begin{fact}\label{QIEmbedd}
    A subgroup $H$ of $G$ is undistorted if and only if the inclusion map $\iota:H\hookrightarrow G$ is a quasi-isometric embedding. In particular, if $H\leq K\leq G$ with $\Dist_H^K(n)\simeq \Dist_K^G(n)\simeq n$, then $\Dist_H^G(n)\simeq n$ because the composition of two quasi-isometric embeddings remains a quasi-isometric embedding.
\end{fact}

\begin{fact}\label{AmalgamFactorDist} If $C$ is an undistorted subgroup of two groups $G_1$ and $G_2$, then both $G_1$ and $G_2$ are undistorted in $G_1\ast_C G_2$ by \cite[Theorem 6.8]{dani2024fractional}.
    
\end{fact}
\begin{fact}\label{UndistortedCyclicsInFree}
    Since all finitely generated free groups are hyperbolic, cyclic subgroups thereof are undistorted.
\end{fact}

\section{Diagrams}\label{secDiagr}
 Annular diagrams, also known as Schupp diagrams, are a powerful tool for studying conjugator length in finitely presented groups. We assume the reader is familiar with them, along with van Kampen diagrams, at the level presented in \cite{lyndon2001combinatorial,BRS}. That said, there are several technical aspects of which we make special note.

 First, in a cyclic-HNN-extension $G'$ of $G$ with stable letter $t$, the only relation involving $t$ is of the form $tut^{-1}=v$. Therefore in a reduced annular (or van Kampen) diagram over $G'$, edges labeled $t$ appear only in chains of cells corresponding to that relation, called $t$-corridors. Moreover, in our contexts, the words along each side of the $t$-corridor will equal (and often be identical to) powers of $u$ and $v$, as shown in Figure \ref{fig:tcor}. Accordingly, for all $r\geq 1$ and $G\in \mathcal{A}_r$, diagrams over $\mathcal{F}(G)=\langle \Lambda_0, t_1,\ldots, t_M\mid t_ix_it_i^{-1}=y_i\rangle$ are comprised entirely of $t_i$ corridors for $i=1,\ldots, M$ by Lemma \ref{iterFundGrpPresentation}; one subtlety is that a $t_i$-corridor can begin and end at the boundary of a diagram, or (depending on $G$) on the side of some $t_j$-corridor for $j>i$.

\begin{figure}
    \centering
\tikzset{every picture/.style={line width=0.75pt}} %set default line width to 0.75pt        

\begin{tikzpicture}[x=0.75pt,y=0.75pt,yscale=-1,xscale=1]

%Straight Lines [id:da2653546135902227] 
\draw    (100,100) -- (128,100) ;
\draw [shift={(130,100)}, rotate = 180] [color={rgb, 255:red, 0; green, 0; blue, 0 }  ][line width=0.75]    (10.93,-3.29) .. controls (6.95,-1.4) and (3.31,-0.3) .. (0,0) .. controls (3.31,0.3) and (6.95,1.4) .. (10.93,3.29)   ;
%Straight Lines [id:da6923336345215882] 
\draw    (130,100) -- (178,100) ;
\draw [shift={(180,100)}, rotate = 180] [color={rgb, 255:red, 0; green, 0; blue, 0 }  ][line width=0.75]    (10.93,-3.29) .. controls (6.95,-1.4) and (3.31,-0.3) .. (0,0) .. controls (3.31,0.3) and (6.95,1.4) .. (10.93,3.29)   ;
%Straight Lines [id:da7860598783984063] 
\draw    (180,100) -- (228,100) ;
\draw [shift={(230,100)}, rotate = 180] [color={rgb, 255:red, 0; green, 0; blue, 0 }  ][line width=0.75]    (10.93,-3.29) .. controls (6.95,-1.4) and (3.31,-0.3) .. (0,0) .. controls (3.31,0.3) and (6.95,1.4) .. (10.93,3.29)   ;
%Straight Lines [id:da09832822938310692] 
\draw    (230,100) -- (278,100) ;
\draw [shift={(280,100)}, rotate = 180] [color={rgb, 255:red, 0; green, 0; blue, 0 }  ][line width=0.75]    (10.93,-3.29) .. controls (6.95,-1.4) and (3.31,-0.3) .. (0,0) .. controls (3.31,0.3) and (6.95,1.4) .. (10.93,3.29)   ;
%Straight Lines [id:da210509234895913] 
\draw    (280,100) -- (328,100) ;
\draw [shift={(330,100)}, rotate = 180] [color={rgb, 255:red, 0; green, 0; blue, 0 }  ][line width=0.75]    (10.93,-3.29) .. controls (6.95,-1.4) and (3.31,-0.3) .. (0,0) .. controls (3.31,0.3) and (6.95,1.4) .. (10.93,3.29)   ;
%Straight Lines [id:da5791163914604911] 
\draw    (350,150) -- (349.96,130) ;
%Straight Lines [id:da0426377931542925] 
\draw    (100,100) -- (100,128) ;
\draw [shift={(100,130)}, rotate = 270] [color={rgb, 255:red, 0; green, 0; blue, 0 }  ][line width=0.75]    (10.93,-3.29) .. controls (6.95,-1.4) and (3.31,-0.3) .. (0,0) .. controls (3.31,0.3) and (6.95,1.4) .. (10.93,3.29)   ;
%Straight Lines [id:da6783546585838545] 
\draw    (300,150.1) -- (299.96,130.1) ;
%Straight Lines [id:da733120230242092] 
\draw    (150,100) -- (150,128) ;
\draw [shift={(150,130)}, rotate = 270] [color={rgb, 255:red, 0; green, 0; blue, 0 }  ][line width=0.75]    (10.93,-3.29) .. controls (6.95,-1.4) and (3.31,-0.3) .. (0,0) .. controls (3.31,0.3) and (6.95,1.4) .. (10.93,3.29)   ;
%Straight Lines [id:da8393751891234464] 
\draw    (250,150.19) -- (250,130) ;
%Straight Lines [id:da2280655540131571] 
\draw    (200,100) -- (200,128) ;
\draw [shift={(200,130)}, rotate = 270] [color={rgb, 255:red, 0; green, 0; blue, 0 }  ][line width=0.75]    (10.93,-3.29) .. controls (6.95,-1.4) and (3.31,-0.3) .. (0,0) .. controls (3.31,0.3) and (6.95,1.4) .. (10.93,3.29)   ;
%Straight Lines [id:da7589532292296766] 
\draw    (330,100) -- (350,100) ;
%Straight Lines [id:da23244234864110913] 
\draw    (349.96,149.52) -- (321.96,149.57) ;
\draw [shift={(319.96,149.58)}, rotate = 359.89] [color={rgb, 255:red, 0; green, 0; blue, 0 }  ][line width=0.75]    (10.93,-3.29) .. controls (6.95,-1.4) and (3.31,-0.3) .. (0,0) .. controls (3.31,0.3) and (6.95,1.4) .. (10.93,3.29)   ;
%Straight Lines [id:da9102235473606701] 
\draw    (319.96,149.58) -- (271.96,149.67) ;
\draw [shift={(269.96,149.67)}, rotate = 359.89] [color={rgb, 255:red, 0; green, 0; blue, 0 }  ][line width=0.75]    (10.93,-3.29) .. controls (6.95,-1.4) and (3.31,-0.3) .. (0,0) .. controls (3.31,0.3) and (6.95,1.4) .. (10.93,3.29)   ;
%Straight Lines [id:da8479257875969572] 
\draw    (269.96,149.67) -- (221.96,149.77) ;
\draw [shift={(219.96,149.77)}, rotate = 359.89] [color={rgb, 255:red, 0; green, 0; blue, 0 }  ][line width=0.75]    (10.93,-3.29) .. controls (6.95,-1.4) and (3.31,-0.3) .. (0,0) .. controls (3.31,0.3) and (6.95,1.4) .. (10.93,3.29)   ;
%Straight Lines [id:da8375141501647427] 
\draw    (219.96,149.77) -- (171.96,149.86) ;
\draw [shift={(169.96,149.87)}, rotate = 359.89] [color={rgb, 255:red, 0; green, 0; blue, 0 }  ][line width=0.75]    (10.93,-3.29) .. controls (6.95,-1.4) and (3.31,-0.3) .. (0,0) .. controls (3.31,0.3) and (6.95,1.4) .. (10.93,3.29)   ;
%Straight Lines [id:da277148519420356] 
\draw    (169.96,149.87) -- (121.96,149.96) ;
\draw [shift={(119.96,149.96)}, rotate = 359.89] [color={rgb, 255:red, 0; green, 0; blue, 0 }  ][line width=0.75]    (10.93,-3.29) .. controls (6.95,-1.4) and (3.31,-0.3) .. (0,0) .. controls (3.31,0.3) and (6.95,1.4) .. (10.93,3.29)   ;
%Straight Lines [id:da7671820320868339] 
\draw    (119.96,149.96) -- (99.96,150) ;
%Straight Lines [id:da09762087936947772] 
\draw    (250,100) -- (250,128) ;
\draw [shift={(250,130)}, rotate = 270] [color={rgb, 255:red, 0; green, 0; blue, 0 }  ][line width=0.75]    (10.93,-3.29) .. controls (6.95,-1.4) and (3.31,-0.3) .. (0,0) .. controls (3.31,0.3) and (6.95,1.4) .. (10.93,3.29)   ;
%Straight Lines [id:da9668072536009613] 
\draw    (300,100) -- (300,128) ;
\draw [shift={(300,130)}, rotate = 270] [color={rgb, 255:red, 0; green, 0; blue, 0 }  ][line width=0.75]    (10.93,-3.29) .. controls (6.95,-1.4) and (3.31,-0.3) .. (0,0) .. controls (3.31,0.3) and (6.95,1.4) .. (10.93,3.29)   ;
%Straight Lines [id:da044257031828427884] 
\draw    (350,100) -- (350,128) ;
\draw [shift={(350,130)}, rotate = 270] [color={rgb, 255:red, 0; green, 0; blue, 0 }  ][line width=0.75]    (10.93,-3.29) .. controls (6.95,-1.4) and (3.31,-0.3) .. (0,0) .. controls (3.31,0.3) and (6.95,1.4) .. (10.93,3.29)   ;
%Straight Lines [id:da7543883893141126] 
\draw    (200,150.29) -- (200,130) ;
%Straight Lines [id:da6450990306494833] 
\draw    (150,150.38) -- (150,130) ;
%Straight Lines [id:da05097394120502807] 
\draw    (100,150.48) -- (100,130) ;

% Text Node
\draw (102,118.4) node [anchor=north west][inner sep=0.75pt]  [font=\small]  {$t$};
% Text Node
\draw (152,118.4) node [anchor=north west][inner sep=0.75pt]  [font=\small]  {$t$};
% Text Node
\draw (202,118.4) node [anchor=north west][inner sep=0.75pt]  [font=\small]  {$t$};
% Text Node
\draw (252,118.4) node [anchor=north west][inner sep=0.75pt]  [font=\small]  {$t$};
% Text Node
\draw (302,118.4) node [anchor=north west][inner sep=0.75pt]  [font=\small]  {$t$};
% Text Node
\draw (352,118.4) node [anchor=north west][inner sep=0.75pt]  [font=\small]  {$t\ \ \ \ \cdots $};
% Text Node
\draw (111.25,82.65) node [anchor=north west][inner sep=0.75pt]  [font=\small]  {$v$};
% Text Node
\draw (161,82.4) node [anchor=north west][inner sep=0.75pt]  [font=\small]  {$v$};
% Text Node
\draw (211,82.4) node [anchor=north west][inner sep=0.75pt]  [font=\small]  {$v$};
% Text Node
\draw (261,82.4) node [anchor=north west][inner sep=0.75pt]  [font=\small]  {$v$};
% Text Node
\draw (311,82.4) node [anchor=north west][inner sep=0.75pt]  [font=\small]  {$v$};
% Text Node
\draw (111,152.4) node [anchor=north west][inner sep=0.75pt]  [font=\small]  {$u$};
% Text Node
\draw (161,152.4) node [anchor=north west][inner sep=0.75pt]  [font=\small]  {$u$};
% Text Node
\draw (211,152.4) node [anchor=north west][inner sep=0.75pt]  [font=\small]  {$u$};
% Text Node
\draw (261,152.4) node [anchor=north west][inner sep=0.75pt]  [font=\small]  {$u$};
% Text Node
\draw (311,152.4) node [anchor=north west][inner sep=0.75pt]  [font=\small]  {$u$};
% Text Node
\draw (53.67,118.4) node [anchor=north west][inner sep=0.75pt]  [font=\small]  {$\cdots $};

\end{tikzpicture}

    \caption{$t$-corridor}
    \label{fig:tcor}
\end{figure}
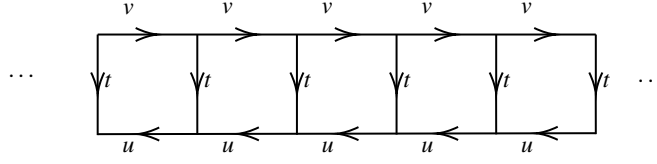

 Second, suppose in an annular diagram $\Omega$ there are two simple, nested, loops $\gamma_1,\gamma_2$ around the inner ``hole" (possibly with non-empty intersection), such that that $\gamma_1,\gamma_2$ do not cross each other anywhere, and that $\gamma_1,\gamma_2$ are labeled by the same words. Then $\gamma_1,\gamma_2$ bound a sub-annulus of $\Omega$ which can be excised, after which ``cutting and stitching" (see \cite{lyndon2001combinatorial} for the precise meaning of this phrase) along $\gamma_1,\gamma_2$ gives a new annular diagram $\Omega'$ with the same words along the boundary components as $\Omega$.
\section{Algebraic Lemmas} In this section we take a purely algebraic, or classically group-theoretic, approach to proving a number of key Lemmas which will be used in our main argument. 
\subsection{Conjugacy in Amalgamated Free Products}

We record three Lemmas that will be of critical use.
\begin{lemma}[Corollary 4.4.1 in \cite{Magnus_Karrass_Solitar}]\label{amalgamNormalForm} Let $G=A\ast_C B$ be an arbitrary amalgamated free product and let $T_A, T_B$ be right transversals of $C$ in $A,B$ respectively. For all $g\in G$, there exist unique $\ell, c\in C, g_1,g_2,\ldots, g_\ell\in T_A\cup T_B$ such that $g=cg_1g_2\cdots g_\ell$ and $g_i,g_{i+1}$ come from different transversals. 
\end{lemma}
For ease of notation, we write  $||g||=\ell$ and note that $||g||=0$ implies $g\in C$, $||g||=1$ implies $g\in (A\cup B)\smallsetminus C$. The following Corollary follows from the proof of \cite[Corollary 4.4.1]{Magnus_Karrass_Solitar}; we make explicit note of it for future reference.

\begin{cor}\label{normalFormFacts}
    For all $g\in A\ast_C B$, the following hold:
    \begin{enumerate}\item \label{compareLengths}$||g||\leq |g|$;\item\label{powerLengths}  if $||g||$ is even and at least 2, then $||g^k||=|k|\cdot ||g||$; \item\label{cycReducedLengths} $g$ is conjugate to some $g'$ such that either $||g'||\leq 1$ or $||g'||$ is even, and such that $c(g,g')\leq |g|$.\end{enumerate} 
\end{cor}

\begin{lemma}[Theorem 4.6 in \cite{Magnus_Karrass_Solitar}]\label{amalgamConjugacyClassification} 
    Let $G=A\ast_CB$ be an arbitrary amalgamated free product, and suppose $g\in G$ is cyclically reduced and has $||g||\in\{0,1\}\cup\{2,4,6,\ldots\}$. Call $A,B$ the \textit{vertex subgroups of $G$}. Then:
    \begin{enumerate}
        \item\label{conjugateToEdgeGroup} If $g$ is conjugate to some $c\in C$, then $||g||\leq 1$ and there exists some sequence $c_1,c_2,\ldots, c_m=c$ of elements of  $cC$ such that (taking $c_0=g$) $c_i$ and $c_{i+1}$ are conjugate in either $A$ or $B$, and the groups in which consecutive $c_i$'s are conjugate alternate.
        \item\label{conjugateToVertexSubgroup} If $g$ is conjugate to some $h\in (A\cup B)\smallsetminus C$, but not conjugate to any element of $C$, then $g$ is in the same vertex subgroup as $h$, and conjugate to $h$ in that subgroup.
        \item\label{conjugateToNeither} If $g$ is conjugate to $h=p_1p_2\cdots p_r$, where the $p_i$'s are from alternating vertex subgroups and $p_1,p_r$ come from different vertex subgroups, then $g$ is conjugate to $h$ by cyclically permuting the $p_i$'s and then conjugating by an element of $C$. In particular, if $g$ is conjugate to $h$ with $||h||\geq 2$ even, then $||g||=||h||$.
    \end{enumerate}
\end{lemma}

\begin{cor}\label{finiteOrderElements}
Let $G$ be as above and let $g\in G$ be a finite order element. Then $g$ is conjugate to some $g'$ with $||g'||\leq 1$, so $g$ is conjugate to an element of $A$ or $B$.
\end{cor}
\begin{proof}
    This follows immediately from Corollary \ref{normalFormFacts}(\ref{powerLengths},\ref{cycReducedLengths}) and Lemma \ref{amalgamConjugacyClassification}\eqref{conjugateToNeither}.
\end{proof}
%For the following lemma, we adopt the notational convention that $ \mathcal{A}_{-1}=\mathcal{A}_0$, and observe that every element of $\mathcal{A}_0$ is (trivially) an amalgamated free product of two elements of $\mathcal{A}_{-1}$ along a cyclic subgroup.
\begin{lemma}\label{distinctPowerConj} For all $r\geq 0$ and $G\in \mathcal{A}_r$, $G$ is torsion free and no powers of a non-trivial element $g\in G$ are conjugate to each other unless their exponents have the same magnitude.

    % \begin{enumerate}
    
    % \item\label{Claim1}For all $u, v\in G$, if there exists $p_0,q_0$ such that $u^{p_0}$ and $v^{q_0}$ are conjugate, then there exists $p,q$ such that $u^p$ is conjugate to $v^q$ and $|q|\leq \lambda|u|$, where $\lambda$ is a constant depending only on $v$. 
    % \item\label{Claim2} No powers of $C$ are conjugate to each other.
    % \end{enumerate}
\end{lemma}
\begin{proof}
    % We prove both claims simultaneously by induction on $k$.

    We proceed by induction on $r$. For the base case, $G$ is free and our claims follow straightforwardly. For $r\geq 1$,  let $g\in G=A\ast_CB$ with $C$ cyclic, $A,B\in \mathcal{A}_{r-1}$. Our first claim holds by Corollary \ref{finiteOrderElements} because $A$ and $B$ are torsion free. For our second claim, without loss of generality we may replace $g$ with $g'$ as given in Corollary \ref{normalFormFacts}\eqref{cycReducedLengths}. If $||g||>1$, then our claim follows from Corollary \ref{normalFormFacts}\eqref{powerLengths} and Lemma \ref{amalgamConjugacyClassification}\eqref{conjugateToNeither}. Alternatively, suppose that $||g||\leq 1$, so $g\in A\cup B$ (without loss of generality $g\in A$), and that $g^\ell$ is conjugate to $g^k$ in $G$. If neither $g^\ell$ nor $g^k$ is conjugate to an element of $C$, then $g^\ell$ and $g^k$ are conjugate in $A$ by Lemma \ref{amalgamConjugacyClassification}\eqref{conjugateToVertexSubgroup}, so $|\ell|=|k|$ by induction. Thus one, and hence both, of $g^\ell$ and $g^k$ are conjugate to some $\gamma\in C$. If $c$ is the generator of $C$ with $\gamma=c^m$, then Lemma \ref{amalgamConjugacyClassification}\eqref{conjugateToEdgeGroup} implies $\gamma$ can only be conjugate to itself and $c^{-m}$ by induction. Yet, $g^\ell$ being conjugate (in $G$) to $g^k$ means that $g^{k\ell}$ is conjugate to $g^{k^2}$. But $\gamma^{\ell}$ is conjugate to $g^{k\ell}$ and $\gamma^k$ is conjugate to $g^{k^2}$, hence $|k|=|\ell|$ and we are done.

\end{proof}

\subsection{Cyclic Subgroup Distortion in $\mathcal{F}(G)$}
Fix some $G\in \mathcal{A}_r$ and recall we have some presentation $\mathcal{F}(G)=\langle \Lambda_0, t_1,\ldots, t_M\mid t_ix_it_i^{-1}=y_i\rangle$.
\begin{lemma}\label{amalgUndist}
    For all $g\in G$, $\langle g\rangle_{G}$ is undistorted.
\end{lemma}
\begin{proof}
    We proceed by induction on $r$. If $r=0$, this is Fact \ref{UndistortedCyclicsInFree}. For $r\geq 1$, without loss of generality we may replace $g$ with $g'$ as given in Corollary 
    \ref{normalFormFacts}\eqref{cycReducedLengths}. If $||g||\leq 1$ then $g\in A\cup B$ and our claim follows from $\langle g\rangle$ and $C$ is (by induction) undistorted in $A$ and $B$, and from Facts \ref{QIEmbedd} and \ref{AmalgamFactorDist}. If $||g||>1$ is even, we have $|q|||g||=||g^q|| \leq |g^q|\leq |q||g|$, with the equality holding by Corollary \ref{normalFormFacts}\eqref{powerLengths}, the first inequality holding by Corollary \ref{normalFormFacts}\eqref{compareLengths}, and the second inequality holding by definition. Our Lemma follows. 
\end{proof}

\begin{cor}\label{iterUndist}
    For all $i=1,\ldots, M$, $\langle x_i\rangle_{\mathcal{F}(G)},\langle y_i\rangle_{\mathcal{F}(G)}$ are undistorted.
\end{cor}
\begin{proof}
    Let $h\in im(f_G)$ be the element to which $x_i$ is conjugate by Lemma \ref{iterFundGrpPresentation}; it suffices to show $\langle h\rangle_{\mathcal{F}(G)}$ is undistorted. Let $w_k$ be the minimal length word for $h^k$; we know $|w_k|\leq |k||w_1|.$ For the lower bound, let $u_k$ be the minimal length word for $f_G'(h)\in G$, so by Lemma \ref{amalgUndist} $C|k|\leq |u_k|\leq |k||u_1|$ for some constant $C>0$. We see $f_G'(w_k)$, as a word, has length at least $|u_k|$, whence $|w_k|>C'|k|$ for some constant $C'>0$, and the first claim is shown. 

\end{proof}

\subsection{Iterated Fundamental Groupoids Redux}\label{IFGRedux} Finally, we finish detailing the structure of $\mathcal{F}(G)$.
\begin{prop}\label{iterFundConjPowers}
    For all $r>0$ and $G\in \mathcal{A}_r$, $\mathcal{F}(G)$ has a presentation of the form $\langle \Lambda_0, t_1,\ldots, t_M\mid t_ix_it_i^{-1}=y_i\rangle$ where \begin{enumerate} \item$\Lambda_0$ is a finite set and $x_i,y_i$ are words on $\Lambda_0\cup\{t_1,\ldots, t_{i-1}\}$ not representing the identity, \item $\langle x_i\rangle_{\mathcal{F}(G)},\langle y_i\rangle_{\mathcal{F}(G)}$ are both undistorted in $\mathcal{F}(G)$ \item  $\langle\Lambda_0\rangle_{\mathcal{F}(G)}$ is free with free basis $\Lambda_0$, and \item no distinct powers of the same $x_i$ or $y_i$ are conjugate to each other in $\mathcal{F}(G)$ unless their exponents have the same magnitude.\end{enumerate}
\end{prop}

\begin{proof}
     All statements in this Proposition save the last are proved in Lemma \ref{iterFundGrpPresentation} and Corollary \ref{iterUndist}. For this item, suppose $\gamma x_i^q\gamma^{-1}=x_i^p$ for some $i$, with $|p|>|q|>0$. Then $\gamma^k x_i^{q^k}\gamma^{-k}=x_i^{p^k}$, and by the second item there exists a constant $C>0$ such that $$|p|^k/C=\abs{p^k}/C\leq \abs{x^{p^k}}_{\mathcal{F}(G)}\leq \abs{\gamma^k x_i^{q^k}\gamma^{-k}}_{\mathcal{F}(G)}\leq 2k|\gamma|_{\mathcal{F}(G)}+|q|^k|x_i|_{\mathcal{F}(G)}.$$ For large $k, |p|^k$ must be greater than $2k|\gamma|_{\mathcal{F}(G)}+|q|^k|x_i|_{\mathcal{F}(G)}$, thus giving a contradiction and implying $|p|=|q|$.

     % the first of which also gives (for all $i\in\{1,\ldots, \ell\}$) some $g_i\in im(f_G)$ to which $y_i$ (and hence $x_i$) is conjugate. If $y_i^p$ is conjugate to $y_i^q$ in $\mathcal{F}(G)$ for $p\neq q$, then $f'_G(g_i^p)=f'_G(g_i)^p$ is conjugate to $f'_G(g_i^q)=f_G'(g_i)^q$ in $G$, so $|p|=|q|$ by Lemma \ref{distinctPowerConj}. The same holds if distinct powers of $x_i$ are conjugate, so we are done.

\end{proof}

 As observed in Remark \ref{CLRetract}, to prove $\CL_G$ is linear as in Corollary \ref{AkCor}, it suffices to bound $\CL_{\mathcal{F}(G)}(n)$. Proposition \ref{iterFundConjPowers} shows that Theorem \ref{SecondMainTheorem} applies to $\mathcal{F}(G)$, and so to that Theorem we now turn.
\section{Proof of Theorem \ref{SecondMainTheorem}}

By the same proof as Lemma \ref{iterFundConjPowers}, a group $H$ satisfying the hypotheses of Theorem \ref{SecondMainTheorem} in fact has the property that no distinct powers of the same $x_i$ or $y_i$ are conjugate to each other in $H$ unless their exponents have the same magnitude. So, to summarize, we have $H=\langle \Lambda_0,t_1,\ldots, t_M\mid t_ix_it_i^{-1}=y_i\rangle$ with the following properties: \begin{enumerate} \item\label{xiyirelators}$\Lambda_0$ is a finite set and $x_i,y_i$ are words on $\Lambda_0\cup\{t_1,\ldots, t_{i-1}\}$ not representing the identity,  \item\label{lambdafree}  $\langle\Lambda_0\rangle_{H}$ is free with free basis $\Lambda_0$, \item\label{noDistortion} $\langle x_i\rangle_{H},\langle y_i\rangle_{H}$ are both undistorted in $H$, and \item \label{distinctpowers}no distinct powers of the same $x_i$ or $y_i$ are conjugate to each other in $H$ unless their exponents have the same magnitude.\end{enumerate}
Recall the existence and construction of cyclically Britton-reduced conjugates $x_i',y_i'$ of $x_i,y_i$ given in Lemma \ref{cannonicalConjugates} -- this construction relied only on the presentation of $\mathcal{F}(G)$, and thus goes through for $H$ without change. (If one of $x_i,y_i$ are conjugate to a cyclically reduced element of $\langle \Lambda_0\rangle,$ define $x_i'$ or $y_i'$ to be that element). Also, for ease of notation let $H_j=\langle \Lambda_0,t_1,\ldots,t_j\rangle_H$ and note (as in Corollary \ref{iterBrittonForm}) that $H_j$ is an HNN extension of $H_{j-1}$ with stable letter $t_j$. Throughout this section, we assume without loss of generality that $x_i$ and $y_i$ are represented by minimal length words in the generators of $H_{i-1}$.

 \begin{defn}
     Consider any $t_i$-corridor in a van Kampen or annular diagram over $H$. An augmented corridor of length $\ell$ is constructed as follows. Let $\alpha_i$, $\beta_i$ be minimal length (in $H_{i-1}$) conjugators taking $x_i$ to $x_i'$ and $y_i$ to $y_i'$ respectively. Then, in the van Kampen diagrams $D_1,D_2$ for $\alpha_i x_i\alpha_i^{-1}(x_i'^{-1})$ and $\beta_iy_i\beta_i^{-1}(y_i')^{-1}$, the individual boundary components labeled by $\alpha_i,x_i'$ and $\beta_i,y_i'$ respectively are each simple paths (the $x_i',y_i'$ components being simple follows from $x_i',y_i'$ being cyclically or cyclically Britton-reduced, as the case may be). Glue two copies of $D_1$ to each other along the boundary components labeled $x_i'$ to obtain a diagram $D_1'$, constructing $D_2'$ from $D_2$ analogously. Then, to a single cell $C$ labeled $t_i x_it_i^{-1}y_i^{-1}$, glue $D_1'$ and $D_2'$ along the boundary components labeled $x_i$ and $y_i$ respectively, obtaining diagram $R_i$. We call the union of $C$ with the copies of $D_1$ and $D_2$ glued to them a ``rung." Finally, glue $\ell$ copies of $R_i$ together along the boundary components labeled $\beta_i^{-1}\beta_it_i\alpha_i^{-1}\alpha_i$. The resulting diagram is called a corridor-lozenge, and the union of all rungs is the augmented corridor. 

 \end{defn} 
\begin{figure}
    \centering

\tikzset{every picture/.style={line width=0.75pt}} %set default line width to 0.75pt        

\begin{tikzpicture}[x=0.8pt,y=0.8pt,yscale=-1,xscale=1]
%uncomment if require: \path (0,366); %set diagram left start at 0, and has height of 366

%Shape: Arc [id:dp37919943202059936] 
\draw  [draw opacity=0] (440.14,225.67) .. controls (415.29,245.96) and (361.88,260) .. (300,260) .. controls (237.77,260) and (184.1,245.8) .. (159.43,225.32) -- (300,200) -- cycle ; \draw   (440.14,225.67) .. controls (415.29,245.96) and (361.88,260) .. (300,260) .. controls (237.77,260) and (184.1,245.8) .. (159.43,225.32) ;  
%Straight Lines [id:da5452983899586896] 
\draw    (303.65,260.1) -- (298.9,259.92) ;
\draw [shift={(296.9,259.85)}, rotate = 2.12] [color={rgb, 255:red, 0; green, 0; blue, 0 }  ][line width=0.75]    (10.93,-3.29) .. controls (6.95,-1.4) and (3.31,-0.3) .. (0,0) .. controls (3.31,0.3) and (6.95,1.4) .. (10.93,3.29)   ;
%Straight Lines [id:da9095463466011044] 
\draw    (253,257.35) -- (249.23,256.86) ;
\draw [shift={(247.25,256.6)}, rotate = 7.43] [color={rgb, 255:red, 0; green, 0; blue, 0 }  ][line width=0.75]    (10.93,-3.29) .. controls (6.95,-1.4) and (3.31,-0.3) .. (0,0) .. controls (3.31,0.3) and (6.95,1.4) .. (10.93,3.29)   ;
%Straight Lines [id:da3738580921394574] 
\draw    (181.25,238.35) -- (176.06,235.94) ;
\draw [shift={(174.25,235.1)}, rotate = 24.9] [color={rgb, 255:red, 0; green, 0; blue, 0 }  ][line width=0.75]    (10.93,-3.29) .. controls (6.95,-1.4) and (3.31,-0.3) .. (0,0) .. controls (3.31,0.3) and (6.95,1.4) .. (10.93,3.29)   ;
%Straight Lines [id:da6691846007043956] 
\draw    (350.5,256.35) -- (346.46,257.19) ;
\draw [shift={(344.5,257.6)}, rotate = 348.23] [color={rgb, 255:red, 0; green, 0; blue, 0 }  ][line width=0.75]    (10.93,-3.29) .. controls (6.95,-1.4) and (3.31,-0.3) .. (0,0) .. controls (3.31,0.3) and (6.95,1.4) .. (10.93,3.29)   ;
%Straight Lines [id:da5410268012675442] 
\draw    (408.75,242.6) -- (405.61,243.86) ;
\draw [shift={(403.75,244.6)}, rotate = 338.2] [color={rgb, 255:red, 0; green, 0; blue, 0 }  ][line width=0.75]    (10.93,-3.29) .. controls (6.95,-1.4) and (3.31,-0.3) .. (0,0) .. controls (3.31,0.3) and (6.95,1.4) .. (10.93,3.29)   ;
%Straight Lines [id:da2207052380273924] 
\draw    (375,254.75) -- (375,215) ;
%Straight Lines [id:da7944406427541871] 
\draw    (325,260) -- (325,215) ;
%Straight Lines [id:da15212423543758813] 
\draw    (275,260) -- (275,215) ;
%Straight Lines [id:da4502971285062687] 
\draw    (225,254.75) -- (225,215) ;
%Straight Lines [id:da8340470256206224] 
\draw    (375,245) -- (375,232) ;
\draw [shift={(375,230)}, rotate = 90] [color={rgb, 255:red, 0; green, 0; blue, 0 }  ][line width=0.75]    (10.93,-3.29) .. controls (6.95,-1.4) and (3.31,-0.3) .. (0,0) .. controls (3.31,0.3) and (6.95,1.4) .. (10.93,3.29)   ;
%Straight Lines [id:da04972939682615951] 
\draw    (325,260) -- (325,232) ;
\draw [shift={(325,230)}, rotate = 90] [color={rgb, 255:red, 0; green, 0; blue, 0 }  ][line width=0.75]    (10.93,-3.29) .. controls (6.95,-1.4) and (3.31,-0.3) .. (0,0) .. controls (3.31,0.3) and (6.95,1.4) .. (10.93,3.29)   ;
%Straight Lines [id:da6626048938748701] 
\draw    (275,250) -- (275,232) ;
\draw [shift={(275,230)}, rotate = 90] [color={rgb, 255:red, 0; green, 0; blue, 0 }  ][line width=0.75]    (10.93,-3.29) .. controls (6.95,-1.4) and (3.31,-0.3) .. (0,0) .. controls (3.31,0.3) and (6.95,1.4) .. (10.93,3.29)   ;
%Straight Lines [id:da08553761476761035] 
\draw    (225,250) -- (225,232) ;
\draw [shift={(225,230)}, rotate = 90] [color={rgb, 255:red, 0; green, 0; blue, 0 }  ][line width=0.75]    (10.93,-3.29) .. controls (6.95,-1.4) and (3.31,-0.3) .. (0,0) .. controls (3.31,0.3) and (6.95,1.4) .. (10.93,3.29)   ;
%Straight Lines [id:da6464966194267039] 
\draw    (145,214.5) -- (455,214.5) ;
%Straight Lines [id:da40827811313940965] 
\draw    (145,185) -- (145,214.7) ;
%Straight Lines [id:da14200473368532118] 
\draw    (145,190) -- (145,203) ;
\draw [shift={(145,205)}, rotate = 270] [color={rgb, 255:red, 0; green, 0; blue, 0 }  ][line width=0.75]    (10.93,-3.29) .. controls (6.95,-1.4) and (3.31,-0.3) .. (0,0) .. controls (3.31,0.3) and (6.95,1.4) .. (10.93,3.29)   ;
%Straight Lines [id:da1332077472597355] 
\draw    (225,185) -- (225,214.7) ;
%Straight Lines [id:da8979110274825698] 
\draw    (225,190) -- (225,203) ;
\draw [shift={(225,205)}, rotate = 270] [color={rgb, 255:red, 0; green, 0; blue, 0 }  ][line width=0.75]    (10.93,-3.29) .. controls (6.95,-1.4) and (3.31,-0.3) .. (0,0) .. controls (3.31,0.3) and (6.95,1.4) .. (10.93,3.29)   ;
%Straight Lines [id:da1945071530182796] 
\draw    (275,185.3) -- (275,215) ;
%Straight Lines [id:da5000025739161482] 
\draw    (275,190.5) -- (275,203.5) ;
\draw [shift={(275,205.5)}, rotate = 270] [color={rgb, 255:red, 0; green, 0; blue, 0 }  ][line width=0.75]    (10.93,-3.29) .. controls (6.95,-1.4) and (3.31,-0.3) .. (0,0) .. controls (3.31,0.3) and (6.95,1.4) .. (10.93,3.29)   ;
%Straight Lines [id:da2684896030116316] 
\draw    (325,185.3) -- (325,215) ;
%Straight Lines [id:da10382417777061059] 
\draw    (325,190.5) -- (325,203.5) ;
\draw [shift={(325,205.5)}, rotate = 270] [color={rgb, 255:red, 0; green, 0; blue, 0 }  ][line width=0.75]    (10.93,-3.29) .. controls (6.95,-1.4) and (3.31,-0.3) .. (0,0) .. controls (3.31,0.3) and (6.95,1.4) .. (10.93,3.29)   ;
%Straight Lines [id:da41296386516989936] 
\draw    (375,185.3) -- (375,215) ;
%Straight Lines [id:da46272701647788217] 
\draw    (375,190.5) -- (375,203.5) ;
\draw [shift={(375,205.5)}, rotate = 270] [color={rgb, 255:red, 0; green, 0; blue, 0 }  ][line width=0.75]    (10.93,-3.29) .. controls (6.95,-1.4) and (3.31,-0.3) .. (0,0) .. controls (3.31,0.3) and (6.95,1.4) .. (10.93,3.29)   ;
%Straight Lines [id:da3182044483534575] 
\draw    (455,185.3) -- (455,215) ;
%Straight Lines [id:da5092746629224171] 
\draw    (455,190.5) -- (455,203.5) ;
\draw [shift={(455,205.5)}, rotate = 270] [color={rgb, 255:red, 0; green, 0; blue, 0 }  ][line width=0.75]    (10.93,-3.29) .. controls (6.95,-1.4) and (3.31,-0.3) .. (0,0) .. controls (3.31,0.3) and (6.95,1.4) .. (10.93,3.29)   ;
%Straight Lines [id:da45356912624326295] 
\draw    (145,214.5) -- (159.43,225.32) ;
\draw [shift={(159.43,225.32)}, rotate = 216.85] [color={rgb, 255:red, 0; green, 0; blue, 0 }  ][line width=0.75]    (0,5.59) -- (0,-5.59)   ;
%Straight Lines [id:da33667262568253276] 
\draw    (440.14,225.67) -- (455,215) ;
\draw [shift={(440.14,225.67)}, rotate = 144.32] [color={rgb, 255:red, 0; green, 0; blue, 0 }  ][line width=0.75]    (0,5.59) -- (0,-5.59)   ;
%Straight Lines [id:da3844469987208321] 
\draw    (159.43,225.32) -- (146.6,215.7) ;
\draw [shift={(145,214.5)}, rotate = 36.85] [color={rgb, 255:red, 0; green, 0; blue, 0 }  ][line width=0.75]    (10.93,-3.29) .. controls (6.95,-1.4) and (3.31,-0.3) .. (0,0) .. controls (3.31,0.3) and (6.95,1.4) .. (10.93,3.29)   ;
%Straight Lines [id:da6990525758509858] 
\draw    (440.14,225.67) -- (453.38,216.17) ;
\draw [shift={(455,215)}, rotate = 144.32] [color={rgb, 255:red, 0; green, 0; blue, 0 }  ][line width=0.75]    (10.93,-3.29) .. controls (6.95,-1.4) and (3.31,-0.3) .. (0,0) .. controls (3.31,0.3) and (6.95,1.4) .. (10.93,3.29)   ;

%Shape: Arc [id:dp9147879925574375] 
\draw  [draw opacity=0] (440.14,173.83) .. controls (415.29,153.54) and (361.88,139.5) .. (300,139.5) .. controls (237.77,139.5) and (184.1,153.7) .. (159.43,174.18) -- (300,199.5) -- cycle ; \draw   (440.14,173.83) .. controls (415.29,153.54) and (361.88,139.5) .. (300,139.5) .. controls (237.77,139.5) and (184.1,153.7) .. (159.43,174.18) ;  
%Straight Lines [id:da9215676361537847] 
\draw    (303.65,139.4) -- (298.9,139.58) ;
\draw [shift={(296.9,139.65)}, rotate = 357.88] [color={rgb, 255:red, 0; green, 0; blue, 0 }  ][line width=0.75]    (10.93,-3.29) .. controls (6.95,-1.4) and (3.31,-0.3) .. (0,0) .. controls (3.31,0.3) and (6.95,1.4) .. (10.93,3.29)   ;
%Straight Lines [id:da39720110779455575] 
\draw    (253,142.15) -- (249.23,142.64) ;
\draw [shift={(247.25,142.9)}, rotate = 352.57] [color={rgb, 255:red, 0; green, 0; blue, 0 }  ][line width=0.75]    (10.93,-3.29) .. controls (6.95,-1.4) and (3.31,-0.3) .. (0,0) .. controls (3.31,0.3) and (6.95,1.4) .. (10.93,3.29)   ;
%Straight Lines [id:da4916207697388614] 
\draw    (181.25,161.15) -- (176.06,163.56) ;
\draw [shift={(174.25,164.4)}, rotate = 335.1] [color={rgb, 255:red, 0; green, 0; blue, 0 }  ][line width=0.75]    (10.93,-3.29) .. controls (6.95,-1.4) and (3.31,-0.3) .. (0,0) .. controls (3.31,0.3) and (6.95,1.4) .. (10.93,3.29)   ;
%Straight Lines [id:da5254988136790071] 
\draw    (350.5,143.15) -- (346.46,142.31) ;
\draw [shift={(344.5,141.9)}, rotate = 11.77] [color={rgb, 255:red, 0; green, 0; blue, 0 }  ][line width=0.75]    (10.93,-3.29) .. controls (6.95,-1.4) and (3.31,-0.3) .. (0,0) .. controls (3.31,0.3) and (6.95,1.4) .. (10.93,3.29)   ;
%Straight Lines [id:da6676936581176159] 
\draw    (408.75,156.9) -- (405.61,155.64) ;
\draw [shift={(403.75,154.9)}, rotate = 21.8] [color={rgb, 255:red, 0; green, 0; blue, 0 }  ][line width=0.75]    (10.93,-3.29) .. controls (6.95,-1.4) and (3.31,-0.3) .. (0,0) .. controls (3.31,0.3) and (6.95,1.4) .. (10.93,3.29)   ;
%Straight Lines [id:da9740128317825817] 
\draw    (375,144.75) -- (375,184.5) ;
%Straight Lines [id:da45458266420471904] 
\draw    (325,139.5) -- (325,184.5) ;
%Straight Lines [id:da15793040838478034] 
\draw    (275,139.5) -- (275,184.5) ;
%Straight Lines [id:da03216258991417764] 
\draw    (225,144.75) -- (225,184.5) ;
%Straight Lines [id:da7664166889769813] 
\draw    (375,154.5) -- (375,167.5) ;
\draw [shift={(375,169.5)}, rotate = 270] [color={rgb, 255:red, 0; green, 0; blue, 0 }  ][line width=0.75]    (10.93,-3.29) .. controls (6.95,-1.4) and (3.31,-0.3) .. (0,0) .. controls (3.31,0.3) and (6.95,1.4) .. (10.93,3.29)   ;
%Straight Lines [id:da9200037643803934] 
\draw    (325,139.5) -- (325,167.5) ;
\draw [shift={(325,169.5)}, rotate = 270] [color={rgb, 255:red, 0; green, 0; blue, 0 }  ][line width=0.75]    (10.93,-3.29) .. controls (6.95,-1.4) and (3.31,-0.3) .. (0,0) .. controls (3.31,0.3) and (6.95,1.4) .. (10.93,3.29)   ;
%Straight Lines [id:da9145054622724733] 
\draw    (275,149.5) -- (275,167.5) ;
\draw [shift={(275,169.5)}, rotate = 270] [color={rgb, 255:red, 0; green, 0; blue, 0 }  ][line width=0.75]    (10.93,-3.29) .. controls (6.95,-1.4) and (3.31,-0.3) .. (0,0) .. controls (3.31,0.3) and (6.95,1.4) .. (10.93,3.29)   ;
%Straight Lines [id:da7183664786556592] 
\draw    (225,149.5) -- (225,167.5) ;
\draw [shift={(225,169.5)}, rotate = 270] [color={rgb, 255:red, 0; green, 0; blue, 0 }  ][line width=0.75]    (10.93,-3.29) .. controls (6.95,-1.4) and (3.31,-0.3) .. (0,0) .. controls (3.31,0.3) and (6.95,1.4) .. (10.93,3.29)   ;
%Straight Lines [id:da5280834167965159] 
\draw    (145,185) -- (159.43,174.18) ;
\draw [shift={(159.43,174.18)}, rotate = 143.15] [color={rgb, 255:red, 0; green, 0; blue, 0 }  ][line width=0.75]    (0,5.59) -- (0,-5.59)   ;
%Straight Lines [id:da425316171560604] 
\draw    (440.14,173.83) -- (455,184.5) ;
\draw [shift={(440.14,173.83)}, rotate = 215.68] [color={rgb, 255:red, 0; green, 0; blue, 0 }  ][line width=0.75]    (0,5.59) -- (0,-5.59)   ;
%Straight Lines [id:da9845420622947827] 
\draw    (159.43,174.18) -- (146.6,183.8) ;
\draw [shift={(145,185)}, rotate = 323.15] [color={rgb, 255:red, 0; green, 0; blue, 0 }  ][line width=0.75]    (10.93,-3.29) .. controls (6.95,-1.4) and (3.31,-0.3) .. (0,0) .. controls (3.31,0.3) and (6.95,1.4) .. (10.93,3.29)   ;
%Straight Lines [id:da19111692798772417] 
\draw    (440.14,173.83) -- (453.38,183.33) ;
\draw [shift={(455,184.5)}, rotate = 215.68] [color={rgb, 255:red, 0; green, 0; blue, 0 }  ][line width=0.75]    (10.93,-3.29) .. controls (6.95,-1.4) and (3.31,-0.3) .. (0,0) .. controls (3.31,0.3) and (6.95,1.4) .. (10.93,3.29)   ;
%Straight Lines [id:da5957165770298839] 
\draw    (145,184.35) -- (455,184.35) ;
%Shape: Arc [id:dp43816637341940456] 
\draw  [draw opacity=0] (145,185) .. controls (145,185) and (145,185) .. (145,185) .. controls (145,129.77) and (214.4,85) .. (300,85) .. controls (385.6,85) and (455,129.77) .. (455,185) -- (300,185) -- cycle ; \draw   (145,185) .. controls (145,185) and (145,185) .. (145,185) .. controls (145,129.77) and (214.4,85) .. (300,85) .. controls (385.6,85) and (455,129.77) .. (455,185) ;  
%Shape: Arc [id:dp42690181024561313] 
\draw  [draw opacity=0] (455,215) .. controls (455,215) and (455,215) .. (455,215) .. controls (455,270.23) and (385.6,315) .. (300,315) .. controls (214.4,315) and (145,270.23) .. (145,215) -- (300,215) -- cycle ; \draw   (455,215) .. controls (455,215) and (455,215) .. (455,215) .. controls (455,270.23) and (385.6,315) .. (300,315) .. controls (214.4,315) and (145,270.23) .. (145,215) ;  

% Text Node
\draw (152,192.4) node [anchor=north west][inner sep=0.75pt]  [font=\scriptsize]  {$t_{i}$};
% Text Node
\draw (232,192.4) node [anchor=north west][inner sep=0.75pt]  [font=\scriptsize]  {$t_{i}$};
% Text Node
\draw (282,192.9) node [anchor=north west][inner sep=0.75pt]  [font=\scriptsize]  {$t_{i}$};
% Text Node
\draw (332,192.9) node [anchor=north west][inner sep=0.75pt]  [font=\scriptsize]  {$t_{i}$};
% Text Node
\draw (382,192.9) node [anchor=north west][inner sep=0.75pt]  [font=\scriptsize]  {$t_{i}$};
% Text Node
\draw (462,192.9) node [anchor=north west][inner sep=0.75pt]  [font=\scriptsize]  {$t_{i}$};
% Text Node
\draw (296,217.4) node [anchor=north west][inner sep=0.75pt]  [font=\scriptsize]  {$x_{i}$};
% Text Node
\draw (346,217.4) node [anchor=north west][inner sep=0.75pt]  [font=\scriptsize]  {$x_{i}$};
% Text Node
\draw (401,217.4) node [anchor=north west][inner sep=0.75pt]  [font=\scriptsize]  {$x_{i}$};
% Text Node
\draw (191,217.4) node [anchor=north west][inner sep=0.75pt]  [font=\scriptsize]  {$x_{i}$};
% Text Node
\draw (246,217.4) node [anchor=north west][inner sep=0.75pt]  [font=\scriptsize]  {$x_{i}$};
% Text Node
\draw (298.9,263.25) node [anchor=north west][inner sep=0.75pt]  [font=\scriptsize]  {$x_{i} '$};
% Text Node
\draw (346.5,261) node [anchor=north west][inner sep=0.75pt]  [font=\scriptsize]  {$x_{i} '$};
% Text Node
\draw (405.75,248) node [anchor=north west][inner sep=0.75pt]  [font=\scriptsize]  {$x_{i} '$};
% Text Node
\draw (192,247.4) node [anchor=north west][inner sep=0.75pt]  [font=\scriptsize]  {$x_{i} '$};
% Text Node
\draw (245,262.4) node [anchor=north west][inner sep=0.75pt]  [font=\scriptsize]  {$x_{i} '$};
% Text Node
\draw (211,232.9) node [anchor=north west][inner sep=0.75pt]  [font=\scriptsize]  {$\alpha $};
% Text Node
\draw (261,232.9) node [anchor=north west][inner sep=0.75pt]  [font=\scriptsize]  {$\alpha $};
% Text Node
\draw (311,232.9) node [anchor=north west][inner sep=0.75pt]  [font=\scriptsize]  {$\alpha $};
% Text Node
\draw (362,232.9) node [anchor=north west][inner sep=0.75pt]  [font=\scriptsize]  {$\alpha $};
% Text Node
\draw (149,222.4) node [anchor=north west][inner sep=0.75pt]  [font=\scriptsize]  {$\alpha $};
% Text Node
\draw (442,222.4) node [anchor=north west][inner sep=0.75pt]  [font=\scriptsize]  {$\alpha $};
% Text Node
\draw (298.9,126.05) node [anchor=north west][inner sep=0.75pt]  [font=\scriptsize]  {$y_{i} '$};
% Text Node
\draw (346.5,128.3) node [anchor=north west][inner sep=0.75pt]  [font=\scriptsize]  {$y_{i} '$};
% Text Node
\draw (405.75,141.3) node [anchor=north west][inner sep=0.75pt]  [font=\scriptsize]  {$y_{i} '$};
% Text Node
\draw (191,141.05) node [anchor=north west][inner sep=0.75pt]  [font=\scriptsize]  {$y_{i} '$};
% Text Node
\draw (245,126.9) node [anchor=north west][inner sep=0.75pt]  [font=\scriptsize]  {$y_{i} '$};
% Text Node
\draw (146,167.05) node [anchor=north west][inner sep=0.75pt]  [font=\scriptsize]  {$\beta $};
% Text Node
\draw (442,167.05) node [anchor=north west][inner sep=0.75pt]  [font=\scriptsize]  {$\beta $};
% Text Node
\draw (246,171.9) node [anchor=north west][inner sep=0.75pt]  [font=\scriptsize]  {$y_{i}$};
% Text Node
\draw (191,171.9) node [anchor=north west][inner sep=0.75pt]  [font=\scriptsize]  {$y_{i}$};
% Text Node
\draw (401,171.9) node [anchor=north west][inner sep=0.75pt]  [font=\scriptsize]  {$y_{i}$};
% Text Node
\draw (346,171.9) node [anchor=north west][inner sep=0.75pt]  [font=\scriptsize]  {$y_{i}$};
% Text Node
\draw (296,171.9) node [anchor=north west][inner sep=0.75pt]  [font=\scriptsize]  {$y_{i}$};
% Text Node
\draw (211,157.4) node [anchor=north west][inner sep=0.75pt]  [font=\scriptsize]  {$\beta $};
% Text Node
\draw (261,157.4) node [anchor=north west][inner sep=0.75pt]  [font=\scriptsize]  {$\beta $};
% Text Node
\draw (311,157.4) node [anchor=north west][inner sep=0.75pt]  [font=\scriptsize]  {$\beta $};
% Text Node
\draw (362,157.4) node [anchor=north west][inner sep=0.75pt]  [font=\scriptsize]  {$\beta $};
% Text Node
\draw (296,63.05) node [anchor=north west][inner sep=0.75pt]  [font=\scriptsize]  {$y_{i}^{5}$};
% Text Node
\draw (296,317.4) node [anchor=north west][inner sep=0.75pt]  [font=\scriptsize]  {$x_{i}^{5}$};

\end{tikzpicture}
    \caption{An augmented $t_i$-corridor of length 5, with the rest of the lozenge shown in less detail}
    \label{fig:augmentedCor}
\end{figure}
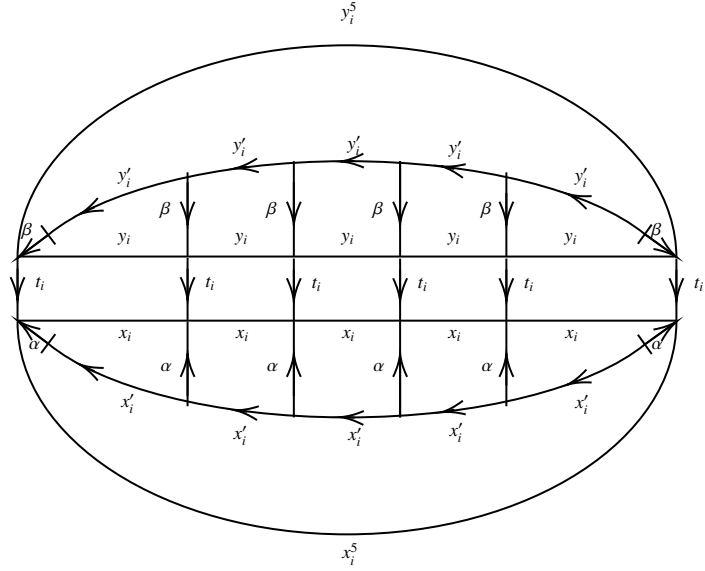
\begin{lemma}\label{augmentedDiagram}
    For every annular or van Kampen diagram $\Omega$, there exists a diagram $\Omega^+$ whose boundary component(s) is (are) labeled by the same word(s) as $\Omega$'s, and such that every $t_i$-corridor $T$ which either begins and ends on $\del\Omega^+$ or forms a loop is contained in an augmented corridor $T^+$. This diagram is not necessarily reduced.
\end{lemma}
\begin{proof}
    Let $S$ be the corridor-lozenge of the same length $\ell_T$ as $T$. If $T$ begins and ends on the boundary of $\Omega$, fold the boundary components of $S$ which are labeled $\alpha_i^{-1}\alpha_i$ and $\beta_i^{-1}\beta_i$, found along the first and last copies of $R_i$, together; otherwise, leave these edges unfolded. Let $p_x$ be the path along the side of $T$ labeled by a power of $x_i$ such that all self-loops are removed, and let $\Delta_x$ be a van Kampen diagram between $x_i^{\ell_T}$ and the label of $p_x$; define $p_y$ and $\Delta_y$ analogously. The boundary component of $S$ labeled $x_i^{\ell_T}$ is simple by construction, so we may glue $\Delta_x$ that side without violating planarity, at least after possible inserting a van Kampen diagram to make $x_i^{\ell_T}$ precisely equal to the label of $p_x$; doing the same with $\Delta_y$ gives a diagram $S'$ whose boundary label(s) is (are) precisely the same of $T$'s. Inserting this $S'$ in place of $T$ in $\Omega$ completes the proof.
\end{proof}
\begin{defn}\label{sides}Let $T^+$ be an augmented $t_i$-corridor $T^+$ beginning and ending on $\del\Omega^+$. We define the \textit{sides} of $T^+$ to be the boundary components of $T^+$ that are not adjacent to the first or last rung; the sides of an augmented $t_i$-corridor forming a loop are the two components of its boundary. In either case, if $\ell_T>2$ they are entirely labeled by $(x_i')^{\ell_T-2}$ and $(y_i')^{\ell_T-2}$.
\end{defn}
\begin{lemma}\label{nonContractibleCorridorsExcision}
    Let $\Omega$ be an annular diagram over $H_j$ such that the boundary words $u,v$ are cyclically Britton-reduced with respect to $t_j$, and suppose $\Omega$ contains at most two  non-contractible $t_j$-corridors $T,S$ with the same number of cells. Then the number of cells in $T\cup S$ is bounded above by $C\lambda$, where $C>0$ is some constant depending on $H_j$ and $\lambda$ is the minimum of the lengths of $u$ and $v$.
\end{lemma}
\begin{proof}
For the entirety of this proof, suppose that $u$ has length $\lambda$, that $T$ is closer to $u$ than $S$, and that the side of $T$ labeled with a power of $x_j$ is facing the boundary component labeled $u$. After replacing $T,S$ with $T^+, S^+$, as constructed in Lemma \ref{augmentedDiagram}, we may take loops $\gamma,\rho$ along the sides of $T^+,S^+$ labeled by $(x_j')^e$, where $e$ is the number of cells in $T$.
    We proceed by induction on $j$. For $j=1$, $(x_j')^e$ and $u$ are cyclically reduced and conjugate words over $\Lambda_0$, hence $e$ equals the length of $u$ divided by the length of $x_j'$. For the inductive step, let $\Omega_1$ be the subdiagram contained between $u$ and $T$, so if $\Omega_1$ has no $t$-edges at all, then $u$ is a cyclic permutation of $(x_j')^e$ and we proceed as in the base case. Otherwise, let $j_1<j$ be the largest index such that there is a $t_{j_1}$-edge in the subdiagram. If there is a $t_{j_1}$ corridor running from $u$ to $\gamma$, then every subword labeled $x_j'$ along $\gamma$ has such a $t_{j_1}$-corridor, since $x_j'$ is cyclically Britton-reduced with respect to $j$; the desired bound follows immediately.

    If there exists no $t_{j_1}$-corridor running from $u$ to $\gamma$, there instead exists at most two non-contractible $t_{j_1}$-corridors $T_1, S_1$ between $u$, and $\gamma$. We claim there are at most two such corridors, each with the same number of cells. Indeed, then the powers of $x_j'$ and $y_j'$ along the sides of any two such corridors must have the same magnitude, by property \eqref{distinctpowers} of $H$, and by annular excision we may remove a subdiagram containing at least one $t_j$-corridor if there are more two (all the corridors must have the same number of cells by \eqref{distinctpowers}, and at least one pair must have the same sign along their sides).
    
    Now, by induction, $T_1, S_1$ have at most $e'\leq C_1\lambda$ cells. Replace them with $T_1^+, S_1^+$. Supposing without loss of generality that that the side of $T_1^+$ labeled by some power $x_{j_1}^{e'}$ is facing $\gamma$, by annular insertion we create a path $\gamma'$ labeled $(x_{j_1}')^{e'}$ looping around $\gamma$. Let $j_2<j_1$ be the largest index such that there is a $t_{j_2}$ edge in the subdiagram $\Omega_2$ between $\gamma$ and $\gamma'$ (like before, if $\Omega_2$ has no $t$-edges at all, then $(x_{j_1}')^{e'}$ is a cyclic permutation of $(x_j')^e$ and we proceed again as in the base case). If there is a $t_{j_2}$ corridor running from $\gamma$ to $\gamma'$, our bound on $e$ follows as in the previous paragraph. Otherwise, we may continue considering $t$-corridors with lower and lower indices, augmenting corridors at each step, until we obtain a loop $\gamma''$, labeled by (say) $(x_{j_k}')^{e''}$ where $e''<C_k\lambda_1$, such that the diagram $\Omega_k$ contained between $\gamma$ and $\gamma''$ is either a diagram over $\langle\Lambda_0\rangle_{H_i}$ or has some $t_{j_k}$-corridor running between $\gamma$ and $\gamma''$. In both cases, the desired bound follows as above, so we are done.
\end{proof}

\begin{proof}[Proof of Theorem \ref{SecondMainTheorem}]
We proceed by induction on the number $M$ of $t$-letters in the presentation of $H$. Let $u,v$ be conjugate words over the generators of $H$ of total length at most $n$. By repeated application of property \eqref{noDistortion} of $H$, at the cost of a constant-factor increase in length and conjugating by a linear-length conjugator, we may assume both $u$ and $v$ are cyclically Britton-reduced with respect to some $t_i$. Without loss of generality, assume $i$ is minimal, so one (and hence both) of $u$ and $v$ contains a $t_i$-letter if $i>0$. Let $\Omega$ be a reduced annular diagram for $u$ and $v$, and $j$ the maximal index of any $t_j$-edge appearing in $\Omega$. Note that, for $j'\leq j$, there is no $t_{j'}$-corridor going to and from the same boundary component of $\Omega$. We have two cases.

For the first case, suppose there exists some non-contractible $t_j$-corridor $T$ for $j>i$. As in Lemma \ref{nonContractibleCorridorsExcision}, by annular excision we may assume there are at most two $t_j$-corridors, and so apply the same lemma to see that they have at most $C\lambda$ many cells, where $\lambda$ is the minimal length of $u$ or $v$. The diagrams between $T$ and $u$, $T$ and $S$, and $S$ and $v$ are all diagrams over $H_{j-1}$, and the claim then follows by induction. 

For the second case, which will comprise the rest of this proof, suppose there exists a $t_j$-corridor $T$ running from one boundary component of $\Omega$ to the other, so $j=i$. Enumerate all $t$-corridors (for any $t_{j'},j'\leq j$) running between the two boundary components as $T_1,T_2,\ldots, T_m$, and assume $T=T_{m+1}=T_1$. If one of these $t$-corridors has less than $3$ cells, we are done. Otherwise, replace $\Omega$ by $\Omega^+$ given by Lemma \ref{augmentedDiagram} and delete the first and last rungs of each $T_k^+$ from $\Omega^+$, resulting in an increase in the length of each boundary component by a constant factor. The only (original) parts remaining of the boundary of $T_k^+$ are its sides, which are not adjacent to the boundary of $\Omega^+$.  

 Let $\ell_k$ be the number of remaining cells in $T_k^+$, and let $\Delta_k$ be the subdiagram contained strictly between $T_k^+$ and $T_{k+1}^+$, so $\Delta_k$ is a van Kampen diagram for $\alpha_k^{\ell_k} u_k\beta_k^{-\ell_{k+1}}v_k^{-1}$, where $\alpha_k,\beta_k\in\{x_1',y_1',\ldots, x_M',y_M'\}$ depend on the orientation and index of $t$-edges in $T_k,T_{k+1}$ and $u_k,v_k$ are subwords of of $u,v$ respectively. Call the segments of $\del\Delta$ labeled $u_k,v_k$ the upper and lower boundary segments, and those labeled $\a_k^{\ell_k},\b_k^{-\ell_{k+1}}$ the left and right contours, respectively. The words upper and lower reflect how we will refer to direction when ``navigating" in $\Delta_k$; when enumerating things such as corridors, higher ones (i.e. ones closer to $u_k$) will be numbered first. Our aim is to bound $\min_k\ell_k$. Let $r_k$ be the largest index such that $t_{r_k}$ appears in either $\alpha_k$ or $\beta_k$, or zero if these words contain no $t$-letters. Then $\Delta_k$ is a van Kampen diagram over $H_{r_k}$, and (if $r_k>0$) every $t_{r_k}$-edge in $\Delta_k$ is part of a $t_{r_k}$-corridor. 
 
 When $r_k>0$, there are no $t_{r_k}$-corridors going to and from $u_k$ or to and from $v_k$ because they are cyclically Britton-reduced. Every $t_{r_k}$-corridor thus has at least one endpoint in $\alpha_k^{\ell_k}$ or $\beta_k^{-\ell_{k+1}}$. If no $t_{r_k}$-corridor runs between the left and right contours, we see $\ell_k+\ell_{k+1}$ is bounded above by $|u_k|+|v_k|\leq |u|+|v|$, and we are done. Otherwise, we call all such $t_{r_k}$-corridors ``bars," and all other $t_{r_k}$-corridors ``boundary corridors." Let $L_k=\lcm(\#_{{r_k}}\alpha_k,\#_{r_k}\beta_k)$ and $\a_k'$ be the cyclic permutation of $\a_k$ starting from the first $t_{r_k}$-edge of the left contour which is \textit{not} is part of a boundary corridor, and define $\b_k'$ analogously. For any segment $S$ of the left contour labeled $(\alpha_k')^{L_k}$ such that every $t_{r_k}$-corridors with an endpoint in $S$ is a bar, there exists a segment $S'$ of the right contour labeled ${\b_k'}^{L_k}$ such that every $t_{r_k}$-corridor with an endpoint in $S\cup S'$ has both endpoints in that set. Supposing instead that $r_k=0$,  we recall that $\alpha_k,\beta_k$ are cyclically freely reduced because they are cyclically Britton-reduced with respect to $t_i$, and define $L_k=\lcm(|\alpha_k|,|\beta_k|)$. Let $\alpha_k'$ be the cyclic permutation of $\alpha_k$ starting from the first letter (reading from $u$ to $v$ along the left contour) of $\alpha_k^{\ell_k}$ which is part of both contours; we define $\beta_k'$ analogously. 
 
 Define $L=\lcm_kL_k$. In the rest of this proof we assume $\min_k\ell_k>|u|+|v|+2L$. If $r_k>0$, the previous paragraph shows there is a bar starting at the first edge of each consecutive left contour segment which is labeled $(\alpha_k')^{L_k}$ and has no boundary corridors, and running to the first edge of each consecutive right contour segment which is labeled $(\b_k')^{L_k}$ and has no boundary corridors. In this case, call the first bar running from the first edge of an $(\alpha_k')^{L}$ segment to the first edge of a $(\b_k')^L$ segment a ``top bar." . Alternatively, if $r_k=0$, then $\Delta_k$ is a diagram over $\langle \Lambda_0\rangle_H$, so by property \eqref{lambdafree} and appealing to cancellation in free groups we see $(\alpha_k')^{L_k}=(\beta_k')^{-L_k}$, so $(\alpha_k')^{L}=(\beta_k')^{-L}$. Call an edge in the left contour of $\Delta_k$ a top edge if it is the first letter of both an $(\alpha_k')^{L}$ contour segment and a $(\beta_k')^{L}$ contour segment; consecutive top edges are top edges of consecutive instances of these words.

 We have two subcases:
 \begin{enumerate}\begin{figure}
         \centering

\tikzset{every picture/.style={line width=0.75pt}} %set default line width to 0.75pt        

\begin{tikzpicture}[x=0.75pt,y=0.75pt,yscale=-1,xscale=1]
%uncomment if require: \path (0,412); %set diagram left start at 0, and has height of 412

%Straight Lines [id:da9510800455724133] 
\draw    (150,50) -- (500,50) ;
%Straight Lines [id:da5446456968239016] 
\draw    (150,350) -- (500,350) ;
%Shape: Arc [id:dp9787791547327369] 
\draw  [draw opacity=0] (210,50) .. controls (210,50) and (210,50) .. (210,50) .. controls (210,77.61) and (183.14,100) .. (150,100) -- (150,50) -- cycle ; \draw   (210,50) .. controls (210,50) and (210,50) .. (210,50) .. controls (210,77.61) and (183.14,100) .. (150,100) ;  
%Shape: Arc [id:dp5069136456025221] 
\draw  [draw opacity=0] (240,50) .. controls (240,88.66) and (199.71,120) .. (150,120) -- (150,50) -- cycle ; \draw   (240,50) .. controls (240,88.66) and (199.71,120) .. (150,120) ;  
%Straight Lines [id:da816373916003121] 
\draw    (150,150) -- (500,100) ;
%Straight Lines [id:da0007246482462345272] 
\draw    (150,170) -- (500,120) ;
%Straight Lines [id:da3500461648196882] 
\draw    (150,300) -- (500,250) ;
%Straight Lines [id:da762510923113903] 
\draw    (150,320) -- (500,270) ;
%Shape: Arc [id:dp2936347730239611] 
\draw  [draw opacity=0] (439.67,349.86) .. controls (439.67,349.86) and (439.67,349.86) .. (439.67,349.86) .. controls (439.54,322.25) and (466.3,299.73) .. (499.44,299.58) -- (499.67,349.58) -- cycle ; \draw   (439.67,349.86) .. controls (439.67,349.86) and (439.67,349.86) .. (439.67,349.86) .. controls (439.54,322.25) and (466.3,299.73) .. (499.44,299.58) ;  
%Shape: Arc [id:dp7744187165167039] 
\draw  [draw opacity=0] (409.67,350) .. controls (409.67,350) and (409.67,350) .. (409.67,350) .. controls (409.49,311.34) and (449.64,279.81) .. (499.34,279.58) -- (499.67,349.58) -- cycle ; \draw   (409.67,350) .. controls (409.67,350) and (409.67,350) .. (409.67,350) .. controls (409.49,311.34) and (449.64,279.81) .. (499.34,279.58) ;  
%Straight Lines [id:da44494972630698937] 
\draw  [dash pattern={on 0.84pt off 2.51pt}]  (140,170) -- (140,250) ;
\draw [shift={(140,250)}, rotate = 270] [color={rgb, 255:red, 0; green, 0; blue, 0 }  ][line width=0.75]    (0,5.59) -- (0,-5.59)   ;
\draw [shift={(140,170)}, rotate = 270] [color={rgb, 255:red, 0; green, 0; blue, 0 }  ][line width=0.75]    (0,5.59) -- (0,-5.59)   ;
%Straight Lines [id:da8033341914780194] 
\draw  [dash pattern={on 0.84pt off 2.51pt}]  (510,120) -- (510,200) ;
\draw [shift={(510,200)}, rotate = 270] [color={rgb, 255:red, 0; green, 0; blue, 0 }  ][line width=0.75]    (0,5.59) -- (0,-5.59)   ;
\draw [shift={(510,120)}, rotate = 270] [color={rgb, 255:red, 0; green, 0; blue, 0 }  ][line width=0.75]    (0,5.59) -- (0,-5.59)   ;
%Shape: Rectangle [id:dp1872536094651871] 
\draw   (100,50) -- (150,50) -- (150,350) -- (100,350) -- cycle ;
%Straight Lines [id:da06030714456809372] 
\draw    (150,230) -- (500,180) ;
%Straight Lines [id:da9703218982406223] 
\draw    (150,250) -- (500,200) ;
%Shape: Rectangle [id:dp09651691289474762] 
\draw   (499.67,49.58) -- (549.67,49.58) -- (549.67,349.58) -- (499.67,349.58) -- cycle ;

% Text Node
\draw (114,32.4) node [anchor=north west][inner sep=0.75pt]  [font=\scriptsize]  {$t_{2}$};
% Text Node
\draw (114,330.4) node [anchor=north west][inner sep=0.75pt]  [font=\scriptsize]  {$t_{2}$};
% Text Node
\draw (521,32.4) node [anchor=north west][inner sep=0.75pt]  [font=\scriptsize]  {$t_{2}$};
% Text Node
\draw (521,330.4) node [anchor=north west][inner sep=0.75pt]  [font=\scriptsize]  {$t_{2}$};
% Text Node
\draw (151,102.4) node [anchor=north west][inner sep=0.75pt]  [font=\scriptsize]  {$t_{1}$};
% Text Node
\draw (152,153.4) node [anchor=north west][inner sep=0.75pt]  [font=\scriptsize]  {$t_{1}$};
% Text Node
\draw (151,307.05) node [anchor=north west][inner sep=0.75pt]  [font=\scriptsize]  {$t_{1}$};
% Text Node
\draw (487,252.4) node [anchor=north west][inner sep=0.75pt]  [font=\scriptsize]  {$t_{1}$};
% Text Node
\draw (487,102.4) node [anchor=north west][inner sep=0.75pt]  [font=\scriptsize]  {$t_{1}$};
% Text Node
\draw (486.34,281.98) node [anchor=north west][inner sep=0.75pt]  [font=\scriptsize]  {$t_{1}$};
% Text Node
\draw (105,202.4) node [anchor=north west][inner sep=0.75pt]  [font=\scriptsize]  {$( \alpha _{k} ')^{L}$};
% Text Node
\draw (511,152.4) node [anchor=north west][inner sep=0.75pt]  [font=\scriptsize]  {$( \beta _{k} ')^{L}$};
% Text Node
\draw (301,152.4) node [anchor=north west][inner sep=0.75pt]  [font=\scriptsize]  {$x_{1}^{p}$};
% Text Node
\draw (301,302.4) node [anchor=north west][inner sep=0.75pt]  [font=\scriptsize]  {$x_{1}^{r}$};
% Text Node
\draw (152,233.4) node [anchor=north west][inner sep=0.75pt]  [font=\scriptsize]  {$t_{1}$};
% Text Node
\draw (487,182.4) node [anchor=north west][inner sep=0.75pt]  [font=\scriptsize]  {$t_{1}$};
% Text Node
\draw (301,232.4) node [anchor=north west][inner sep=0.75pt]  [font=\scriptsize]  {$x_{1}^{p}$};

\end{tikzpicture}

         \caption{Three consecutive top corridors}
         \label{fig:topCors}
     \end{figure}
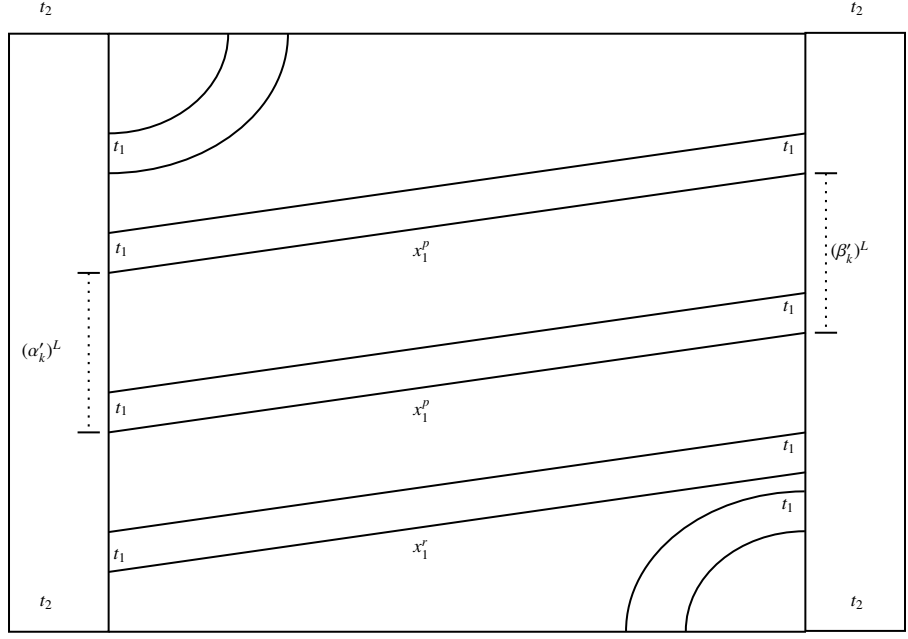
     \item Suppose, for some $k$ with $r_k>0$, that there are consecutive top bars with $p,q,$ and $r$ cells respectively, such that $p\neq q$. Note that bars are not themselves augmented corridors. The situation is as depicted in Figure \ref{fig:topCors} ($r_k$ is taken to be 2 in this instance for simplicity). Letting $\gamma_1=(\alpha_k')^{L},\gamma_2=(\b_k')^{L}$ for ease of notation, we have $\gamma_1^{-1} x_{r_k}^p\gamma_2 = x_{r_k}^q$ and $\gamma_1^{-1} x_{r_k}^q\gamma_2 = x_{r_k}^r$. This second equation is equivalent to $x_{r_k}^r=\gamma_1^{-1} x_{r_k}^q\gamma_2=\gamma_1^{-1} x_{r_k}^{q-p}\gamma_1\gamma_1^{-1}x_{r_k}^p\gamma_2=\gamma_1^{-1} x_{r_k}^{q-p}\gamma_1x_{r_k}^q$. Since conjugate powers of $x_{r_k}$ must have exponents with the same magnitude by property \eqref{distinctpowers} of $H$, this implies $|r-q|=|q-p|$. 

    Starting from the top corridor closest to $u_k$, enumerate the top corridors $S_1$ through $S_N$, and let $\lambda_1,\rho_1$ be the paths from the start and endpoints of $u_k$ to the start and endpoints of $S_1$; analogously, let $\lambda_2,\rho_2$ be the paths from the start and endpoints of $S_N$ to the start and endpoints of $v_k$. Every $t_{r_k}$-edge in $\lambda_1,\rho_1,\lambda_2,\rho_2$ is part of a boundary corridor, so their total length is at a constant multiple of the maximal length of $\alpha_k,\beta_k$. In particular, the number of cells in both $S_1$ and $S_N$ are at most some constant times $\sigma$, since $\langle x_{r_k}\rangle$ is undistorted in $H_{j}$. The previous paragraph shows that the number of cells either increases in an arithmetic progression, or switches between two values on consecutive bars. If the the first case holds for any $k$, we obtain a linear bound for $N$. The number of cells in $T_k$ is linear in $LN$, and so we are done. If the second case always holds, a minor modification of the next case suffices.   
\item Suppose, for all $k$ with $r_k>0$, that all top bars have the same number of cells. We create a loop $\lambda_1$ as follows. Start at the beginning of the first $\alpha_1$ instance in the left contour of $\Delta_1$ containing top edge or bar $B_1$, then move to the beginning of $B_1$ and follow it to $T_2^+$. Move along the side of $T_2^+$ down to the beginning of the next instance of $\alpha_1$ (\textit{not} $\alpha_1'$ or $(\alpha_1')^L$). Cross $T_2^+$ along the top of this instance's rung, and move along the side of $T_2^+$ to the beginning of the first $\alpha_2$ instance in the left contour of $\Delta_2$ containing top edge or bar $B_2$, and continue this path along each $\Delta_k$ until a loop is formed.

Replacing all instances of ``first" in this construction with ``second" will give another loop $\gamma_2$: the paths along the bars are identical by assumption, and the paths up and down the sides of augmented $t$-corridors are translates of those in $\gamma_1$ by exactly $L$-cells, and thus \textit{also} have the same label. Therefore, the region between $\gamma_1$ and $\gamma_2$ may be excised, and we have shortened all augmented $t$-corridors. This implies there is a minimal length conjugator taking $u$ to $v$ of length at most $C(|u|+|v|)$, and we are done. Our proof is thus complete.
     
 \end{enumerate}\end{proof}

\appendix

\section{Conjugator Length of Generalized Baumslag-Solitar Groups}

In this appendix, $\Gamma$ is the base graph of a graph of groups $G$ (defined with respect to a fixed spanning tree $T\subseteq \Gamma$) where every vertex- and edge-group is $\ZZ$. The elements $a_e,b_e\in G$ will generate the images of the inclusion maps $H_e\to G_{\alpha(e)},H_e\to G_{\omega(e)}$, as in Subsection \ref{FundGrps}. For ease of notation, the generator of a vertex group $G_v$ ($v\in V(\Gamma)$) shall simply be identified with $v$.

\begin{lemma}\label{AppLemma}
    Either $\langle v\rangle_G$ is undistorted in $G$ for all $v\in V(\Gamma)$, or for all $v\in V(\Gamma)$, there exists a $g_v\in G$ such that $\langle v,g_v\rangle_G\simeq \BS(n_1,n_2)$ where $|n_1|\neq |n_2|$ depend on $v$, and hence $\langle v\rangle_G$ is exponentially distorted.
\end{lemma}
\begin{proof}
    That each $\langle v\rangle_G$ is at most exponentially distorted holds because removing all $t_e$-pinches for $e\not\in E(T)$ increases the length of a word exponentially and gives an element of the restriction of $G$ to $T$, wherein (by Lemma \ref{amalgUndist}) no cyclic subgroups are distorted. If $G$ is unimodular, it is virtually $F_n\times \ZZ$, so all cyclic subgroups are undistorted. If it is not unimodular, by \cite[Lemma 2.4(1)]{Levitt2007GBS} there exists some $v_0$ and $g_{v_0}$ such that $\langle v_0,g_{v_0}\rangle_G\simeq \BS(n_1,n_2)$. Every other vertex in $V(\Gamma)$ has a power that is equal to a power of $v_0$, so we are done.
\end{proof}

We say an element of $G$ is elliptic if it is conjugate to an element of some $G_v$, hyperbolic otherwise; note that the elements of one conjugacy class must either be all elliptic or all hyperbolic. Also, any word representing an elliptic ellement can be conjugated to an element of $G_v$ by a linear-length conjugator by the normal form for fundamental group(oid)s of graphs. To compute $
\CL_G(n)$, we pass to the fundamental groupoid of $\Gamma$, which by an abuse of notation we also write as $\mathcal{F}(G)$; the terms elliptic and hyperbolic port to this group in the obvious way. In either case, it can be shown by standard Bass-Serre theoretic techniques that cyclic subgroups generated by hyperbolic elements are undistorted.

\begin{proof}[Proof of the Hyperbolic Case]
    Let $w_1,w_2$ be words over the generators of $G$ of total length $n$ representing conjugate elements, and let $\Omega$ be a reduced annular diagram for $\iota(w_1),\iota(w_2)$ over $H$. Cyclically permute both so that all $t$-arches from $\Omega$ form $t$-pinches in $w_1,w_2$, and excise such $t$-pinches. If there are no radial $t$-corridors in $\Omega$, then this excision gives words over $V(\Gamma)$ which are freely conjugate (there are no other cells in $\Omega$). The conjugator, moreover, consists of a subword of $w_1,$ a power $v^
\ell$ of some $v
    \in V(\Gamma)$, and a subword of $w_2$. Moreover, if $G$ is unimodular than $\ell$ is linear in $|w_1|+|w_2|,$ otherwise it is exponential in the same; our claim follows by the previous lemma and the same logarithmic-length-shortenning technique as in the $\BS(1,m)$ argument of \cite{BRS}. If instead there is a radial corridor, the same argument as the second case in the proof of Theorem \ref{SecondMainTheorem}, in the case where $r_k=0$ for all $k$, gives that $w_1$ and $w_2$ are (up to cyclic permutation) conjugate by some $v^{\ell}$, where $\ell$ is again linear or exponential in $|w_1|+|w_2|$ if $G$ is (respectively) unimodular or not. The claim follows as above.
\end{proof}

\begin{rem}
    Along with our discussion of distortion, Wei\ss's  algorithm given in \cite[Proposition 28]{Weiss} also suffices to prove this case.   
\end{rem}
The argument for the elliptic case is substantially due to \cite{Weiss} and sources cited therein, particularly in its translation of conjugacy in $G$ to a congruence on $\NN^m\times \NN^{|V(\Gamma)|}$ and its appeal to techniques for semilinear sets -- for the reduction to a congruence on $\NN^m$ we are not aware of a direct source.
\begin{proof}[Proof of Elliptic Case] Let $\mathcal{P}=\{p_1,\ldots, p_m\}$ be the set of primes dividing one of the powers of a vertex equaling some $a_e$, including -1 as a prime. We can represent any element of the form $v^k$ ($v\in V(\Gamma)$) with the tuple $(r_k, e_1,e_2,\ldots,e_m, z_v)$ where $r_k$ is the largest divisor of $k$ not containing any element of $\mathcal{P}$, $e_i$ is the largest power of $p_i$ dividing $k$, and $z_v$ is the vector in $\NN^{|V(\Gamma)|}$ with a 1 in the spot corresponding to $v$ and 0's elsewhere. One can show (\cite{Weiss}) that $v^k$ and $u^
\ell$ are conjugate only if $r_k=r_\ell$; we accordingly suppress the first coordinate of our tuple and represent $v^k$ by $(e_1,e_2,\ldots,e_m, z_v)\in \NN^m\times \NN^{|V(\Gamma)|}$. Slightly varying from \cite{Weiss} in order to maintain consistency with our own notation, we abbreviate $(e_1,e_2,\ldots,e_m)$ as $E(k)$. For a vector $f=(f_1,f_2,\ldots, f_m)\in \NN^m$, define $N(f)=\prod_{i=1}^m p_i^{f_i}$ and note $N(E(k))=k/r_k$. Here are two crucial facts about $N$ we shall use -- the first is trivial, and the second follows immediately from the first.
\begin{itemize}
    \item For $f,g\in \NN^m$, $N(f+g)=N(f)N(g)$,
    \item For $f,g,h\in \NN^m$, if $\gamma$ conjugates $v^{N(f)}$ to $u^{N(g)}$, then it also conjugates $v^{N(f+h)}$ to $u^{N(g+h)}$
\end{itemize}

Wei$\ss$ defines the relation $\sim$ on $\NN^m\times \{z_v\mid v\in V(\Gamma)\}$ by declaring $(f,z_v)\sim(g,z_u)$ if and only if  $v^{N(f)}$ is conjugate to $u^{N(g)}$, and extends it to a congruence on all of $\NN^m\times \NN^{|V(\Gamma)|}$ in the trivial way. Here we study $\sim$ on $\NN^m\times \{z_v\mid v\in V(\Gamma)\}$ in greater detail. Let $e$ be an edge in $\Gamma$ between $v$ and $u$, and suppose the corresponding relation in $\mathcal{F}(G)$ is $t_ev^{A_e}t_e^{-1}=u^{B_e}$ for some integers $A_e,B_e$. Declare $(E(A_e)+g, z_v)\approx(E(B_e)+g,z_u)$ for all $g\in \NN^m$, and let $\approx^*$ be the transitive closure of $\approx$. By the Conjugacy Criterion \cite{horadam1981word}, $(f,z_v)\approx^*(g,z_u)$ if and only if $v^{N(f)}$ is conjugate to $u^{N(g)}$.

Wei$\ss$ \cite{Weiss} showed that $\sim$ is a semilinear relation on $\NN^m\times \NN^{|V(\Gamma)|}$, a fact which we now use to define semilinear relations on just $\NN^m$. For $v,u\in V(\Gamma)$, let $R_{v,u}=\{(f,g)\mid (f,z_v)\approx^*(g,z_u)\}\subseteq 'NN^m\times \NN^m$. This is found by intersecting the graph of $\approx^*$ (or, equivalently, $\sim$) with the semilinear set $(\NN^m\times\{z_v\})\times(\NN^m\times \{z_u\})$, then projecting to the $\NN^m$ factor in each coordinate, that is, forgetting the $z_v,z_u$ unit vectors. Semilinear sets are closed under projection and intersection by a classical result of Ginsburg-Spanier \cite{GinsburgSpanier1966}, hence $R_{v,u}$ is semilinear.

By the definition of semilinear, $R_{v,u}$ is thus the finite union $\bigcup_i S_i^{v,u}$ where $S_i^{v,u}$ is of the form $\{c_i^{v,u}+\sum_j t_jp_{i,j}^{v,u}\mid t_j\in \NN\}$ for $c_i^{v,u}\in R_{v,u}$ and non-zero $p_{i,j}^{v,u}\in \NN^m\times \NN^m$. To finish our conjugator length bound, we require only one more claim:
\begin{claim} If $(x,y)\in R_{v,u}$ then $(y,y)+p_{i,j}^{v,u}\in R_{u,u}$ for all $p_{i,j}^{v,u}$. 
\end{claim}

Write $p_{i,j}^{v,u}=(\pi,\pi')$. We automatically have $(x,y)+p_{i,j}^{v,u}\in R_{v,u}$, so $v^{N(x+\pi)}$ is conjugate to $u^{N(y+\pi')}$. But, $v^{N(x+\pi)}=v^{N(x)N(\pi)}$ is conjugate to $u^{N(y)N(\pi)}=u^{N(y+\pi)}$. The claim follows from the definitions of $R_{u,u}$ and $\approx^*$.

Now we conclude our proof. Write $c_{i}^{v,u}=(\gamma_i^{v,u},(\gamma_i^{v,u})')$ and $p_{i,j}^{v,u}=(\pi_{i,j}^{v,u},(\pi_{i,j}^{v,u})')$, where $(\gamma_i^{v,u})'$ and $(\pi_{i,j}^{v,u})'$ denote specific elements of $\NN^m$ (we separate out the prime solely for clarity's sake). Let $C$ be the maximal length of a minimal conjugator taking any $$v^{N(\gamma_i^{v,u})} \text{  to }u^{N((\gamma_i^{v,u})')}$$ or $$u^{N((\gamma_i^{v,u})'+\pi_{i,j}^{v,u})}\text{  to }u^{N((\gamma_i^{v,u})'+(\pi_{i,j}^{v,u})')}.$$ Note that a conjugator of the second sort exists by our Claim. Let $v^k$ and $u^\ell$ be conjugate non-trivial elements, so $(E(k),E(\ell))\in R_{v,u}$, and without loss of generality suppose $(E(k),E(\ell))\in S_1^{v,u}$. For brevity we abbreviate $c_1^{v,u}$ by $c=(\gamma,\gamma')$ and $p_{1,j}^{v,u}$ by $p_j=(\pi_j,\pi_j')$. Semilinearity gives $(E(k),E(\ell))=c+\sum_j t_j p_j$, so $E(k)=\gamma+\sum_j t_j\pi_j, E(\ell)=\gamma'+\sum_jt_j\pi_j'$. There exists a conjugator of length at most $C$ taking $v^k=v^{r_k N(\gamma)N(\sum_j t_j\pi_j)}$ to $u^{r_k N(\gamma')N(\sum_j t_j\pi_j)}$. Similarly, there exist $\sum_j t_j$ many conjugators of length at most $C$ whose product takes $u^{r_k N(\gamma')N(\sum_j t_j\pi_j)}$ to $u^{r_k N(\gamma')N(\sum_j t_j\pi_j')}=u^\ell$ (the length bound on these conjugators follows from the second fact we noted about $N$ and the second displayed segment of the definition of $C$). Since we are taking a sum of non-zero vectors of natural numbers, $\sum_jt_j$ is at most the sum of all entries of $E(k)+E(\ell)$, which Wei$\ss$ noted is logarithmic in $k+\ell$ \cite{Weiss}. The conjugator length bound then then follows by Lemma \ref{AppLemma}.

\end{proof}
\bibliographystyle{alpha}
\bibliography{bib}

@article{Levitt2007GBS,
  author  = {Levitt, Gilbert},
  title   = {On the automorphism group of generalized {B}aumslag--{S}olitar groups},
  journal = {Geom. Topol.},
  fjournal = {Geometry \& Topology},
  volume  = {11},
  number  = {1},
  pages   = {473--515},
  year    = {2007},
  doi     = {10.2140/gt.2007.11.473},
  eprint  = {math/0511083},
  archivePrefix = {arXiv},
  primaryClass  = {math.GR},
  issn    = {1465-3060},
}

@article{GinsburgSpanier1966,
  author    = {Seymour Ginsburg and Edwin H. Spanier},
  title     = {Semigroups, Presburger formulas, and languages},
  journal   = {Pacific Journal of Mathematics},
  volume    = {16},
  number    = {2},
  pages     = {285--296},
  year      = {1966},
  publisher = {Pacific Journal of Mathematics},
  doi       = {10.2140/pjm.1966.16.285},
  url       = {https://projecteuclid.org/journals/pacific-journal-of-mathematics/volume-16/issue-2/Semigroups-Presburger-formulas-and-languages/pjm/1102994974.full}
}

@article{horadam1981word,
  author     = {Horadam, K. J.},
  title      = {The word problem and related results for graph product groups},
  journal    = {Proceedings of the American Mathematical Society},
  volume     = {82},
  number     = {2},
  pages     = {157--164},
  year      = {1981},
  publisher = {American Mathematical Society},
  issn      = {0002-9939}
}

@article{bezverkhnii2016conjugacy,
  title={Conjugacy word problem in the tree product of free groups with a cyclic amalgamation},
  author={Bezverkhnii, V. N. and Logacheva, E. S.},
  journal={Diskretnaya Matematika},
  volume={28},
  number={1},
  pages={3--18},
  year={2016},
  publisher={Russian Academy of Sciences, Steklov Mathematical Institute of Russian~…}
}

@incollection {Weiss,
    AUTHOR = {Wei\ss, A.},
     TITLE = {A logspace solution to the word and conjugacy problem of
              generalized {B}aumslag-{S}olitar groups},
 BOOKTITLE = {Algebra and computer science},
    SERIES = {Contemp. Math.},
    VOLUME = {677},
     PAGES = {185--212},
 PUBLISHER = {Amer. Math. Soc., Providence, RI},
      YEAR = {2016},
      ISBN = {978-1-4704-2303-2},
   MRCLASS = {20E06 (20F10 68Q17 68R05)},
  MRNUMBER = {3589811},
MRREVIEWER = {Alejandra\ Garrido},
}

@Inbook{Gersten1992,
author="Gersten, S. M.",
editor="Baumslag, G.
and Miller, C. F.",
title="Dehn Functions and l1-norms of Finite Presentations",
bookTitle="Algorithms and Classification in Combinatorial Group Theory",
year="1992",
publisher="Springer New York",
address="New York, NY",
pages="195--224",
isbn="978-1-4613-9730-4",
doi="10.1007/978-1-4613-9730-4_9",
url="https://doi.org/10.1007/978-1-4613-9730-4_9"
}

@article{platonov2004isoparametric,
  title={An isoperimetric function of the {B}aumslag--Gersten group},
  author={Platonov, A. N.},
  journal={Vestnik Moskovskogo Universiteta. Seriya 1. Matematika. Mekhanika},
  number={3},
  pages={12--17},
  year={2004},
  publisher={Lomonosov Moscow State University}
}

@article{BridsonRileySaleMMJ,
  author        = {Bridson, Martin R. and Riley, Timothy R. and Sale, Andrew W.},
  title         = {Conjugacy in a family of free-by-cyclic groups},
  journal       = {Michigan Mathematical Journal},
  year          = {2026},
  note          = {In press. Preprint available at arXiv:2506.01248},
  eprint        = {2506.01248},
  archivePrefix = {arXiv},
  primaryClass  = {math.GR},
  url           = {https://arxiv.org/abs/2506.01248}
}

@article{gillis2025conjugator,
  title={Conjugator Length in the Baumslag-Gersten Group},
  author={Gillis, C.},
  journal={arXiv preprint arXiv:2507.21505},
  year={2025}
}

@article{BAUMSLAGTAYLOR,
author = {Baumslag, G. and Taylor, T.},
journal = {Mathematische Annalen},
pages = {315-319},
title = {The Centre of Groups with One Defining Relator.},
url = {http://eudml.org/doc/161671},
volume = {175},
year = {1968},
}

@book{rotman1999introduction,
  title={An Introduction to the Theory of Groups},
  author={Rotman, J.},
  isbn={9780387942858},
  lccn={94006507},
  series={Graduate Texts in Mathematics},
  url={https://books.google.com/books?id=vb9kUfHqQigC},
  year={1999},
  publisher={Springer New York}
}

@article{HNNOriginal,
author = {Higman, G. and Neumann, B. H. and Neumann, H.},
title = {Embedding Theorems for Groups},
journal = {Journal of the London Mathematical Society},
volume = {s1-24},
number = {4},
pages = {247-254},
doi = {https://doi.org/10.1112/jlms/s1-24.4.247},
url = {https://londmathsoc.onlinelibrary.wiley.com/doi/abs/10.1112/jlms/s1-24.4.247},
eprint = {https://londmathsoc.onlinelibrary.wiley.com/doi/pdf/10.1112/jlms/s1-24.4.247},
year = {1949}
}

@book{trees,
 author = {Serre, J. P.},
 date = {1980},
 doi = {10.1007/978-3-642-61856-7},
 editor = {},
 pages = {},
 title = {Trees},
 url = {https://app.dimensions.ai/details/publication/pub.1044288604},
 year = {1980}
}

@article{kharlampovich1998hyperbolic,
  title={Hyperbolic groups and free constructions},
  author={Kharlampovich, O. and Myasnikov, A.},
  journal={Transactions of the American Mathematical Society},
  volume={350},
  number={2},
  pages={571--613},
  year={1998}
}

@incollection{Bridson6,
	author = {M. R. Bridson},
	booktitle = {Invitations to Geometry and Topology},
	editor = {M. R. Bridson and S. M. Salamon},
	pages = {33-94},
	publisher = {O.U.P.},
	title = {The geometry of the word problem},
	year = {2002}}

@article{Larsen_1977, title={The conjugacy problem and cyclic HNN constructions}, volume={23}, DOI={10.1017/S1446788700019546}, number={4}, journal={Journal of the Australian Mathematical Society}, author={Larsen, L.}, year={1977}, pages={385–401}}

@article{Pietrowski1974,
author = {Pietrowski, A.},
journal = {Mathematische Zeitschrift},
pages = {95-106},
title = {The Isomorphism Problem for One-relator Groups with Non-trivial Centre.},
url = {http://eudml.org/doc/172028},
volume = {136},
year = {1974},
}

@book{lyndon2001combinatorial,
  title={Combinatorial Group Theory},
  author={Lyndon, R. and Schupp, P.},
  isbn={9783540411581},
  lccn={76012537},
  series={Classics in Mathematics},
  url={https://books.google.com/books?id=aiPVBygHi\_oC},
  year={2001},
  publisher={Springer Berlin Heidelberg}
}

@misc{dani2024fractional,
      title={Fractional distortion in hyperbolic groups}, 
      author={P. Dani and T. Riley},
      year={2024},
      eprint={2403.08645},
      archivePrefix={arXiv},
      primaryClass={math.GR}
}

@book{Magnus_Karrass_Solitar, place={Mineola, N.Y}, title={Combinatorial group theory: Presentations of groups in terms of generators and relations}, publisher={Dover Publ}, author={Magnus, W. and Karrass, A. and Solitar, D.}, year={1976}}

@misc{BRS,
      title={Conjugator length in finitely presented groups}, 
      author={Martin R. Bridson and Timothy R. Riley and Andrew W. Sale},
      year={2026},
      eprint={2607.20401},
      archivePrefix={arXiv},
      primaryClass={math.GR},
      url={https://arxiv.org/abs/2607.20401}, 
}

\ni  {Conan Gillis} \rule{0mm}{6mm} \\
Department of Mathematics, 310 Malott Hall,  Cornell University, Ithaca, NY 14853, USA 
\\ {cg527@cornell.edu}, \
\href{https://math.cornell.edu/conan-gillis}{http://www.math.cornell.edu/conan-gillis/}

\end{document}